\documentclass[12pt] {amsart}

\usepackage{amssymb,amsmath,amscd,amsfonts,a4,
psfrag,graphicx,latexsym,epsfig,color}
\usepackage[all]{xy}

\newtheorem{theo}{Theorem}[section]
\newtheorem{defi}[theo]{Definition}

\newtheorem{lem}[theo]{Lemma}
\newtheorem{prop}[theo]{Proposition}
\newtheorem{rem}[theo]{Remark}
\newtheorem{coro}[theo]{Corollary}
\newtheorem{example}[theo]{Example}

\newcommand{\Agot}{\ensuremath{\mathfrak{A}}}

\newcommand{\kgot}{\ensuremath{\mathfrak{k}}}
\newcommand{\hgot}{\ensuremath{\mathfrak{h}}}
\newcommand{\ggot}{\ensuremath{\mathfrak{g}}}
\newcommand{\ngot}{\ensuremath{\mathfrak{n}}}

\newcommand{\tgot}{\ensuremath{\mathfrak{t}}}

\newcommand{\qgot}{\ensuremath{\mathfrak{q}}}

\newcommand{\agot}{\ensuremath{\mathfrak{a}}}
\newcommand{\bgot}{\ensuremath{\mathfrak{b}}}
\newcommand{\Rgot}{\ensuremath{\mathfrak{R}}}
\newcommand{\Sgot}{\ensuremath{\mathfrak{S}}}

\newcommand{\Acal}{\ensuremath{\mathcal{A}}}
\newcommand{\Bcal}{\ensuremath{\mathcal{B}}}

\newcommand{\Fcal}{\ensuremath{\mathcal{F}}}

\newcommand{\Lcal}{\ensuremath{\mathcal{L}}}

\newcommand{\Ocal}{\ensuremath{\mathcal{O}}}
\newcommand{\Pcal}{\ensuremath{\mathcal{P}}}

\newcommand{\Scal}{\ensuremath{\mathcal{S}}}
\newcommand{\Ucal}{\ensuremath{\mathcal{U}}}

\newcommand{\RR}{\ensuremath{\mathrm{RR}}}

\newcommand{\Nbb}{\ensuremath{\mathbb{N}}}
\newcommand{\Cbb}{\ensuremath{\mathbb{C}}}
\newcommand{\Rbb}{\ensuremath{\mathbb{R}}}
\newcommand{\Zbb}{\ensuremath{\mathbb{Z}}}

\newcommand{\T}{\ensuremath{\hbox{\bf T}}}

\newcommand{\ind}{\mathrm{Ind}}

\def \wR {{\widehat{R}}}

\def \clif {\mathbf{c}}

\date{January 2024}

\author[Paul-\'Emile Paradan]{Paul-\'Emile Paradan}
\address{IMAG, Univ Montpellier, CNRS, Montpellier, France} 
\email{paul-emile.paradan@umontpellier.fr}

\title{Symmetric pairs and branching laws}

\begin{document}

\begin{abstract}
Let $G$ be a compact connected Lie group and let $H$ be a subgroup fixed by an involution. A classical result assures that the $H_\Cbb$-action on 
the flag variety $\Fcal$ of $G$ admits a finite number of orbits. In this article we propose a formula for the branching coefficients of the symmetric pair $(G,H)$ that is parametrized by $H_\Cbb\backslash \Fcal$.
\end{abstract}

\maketitle

\tableofcontents

\section{Introduction}

Let $G$ be a compact connected Lie group equipped with an involution $\theta$. Let  $G^{\theta}:=\{g\in G, \theta(g)=g\}$ be the subgroup 
fixed by the involution. We consider a subgroup $H\subset G$ such that $(G^{\theta})_0\subset H\subset G^{\theta}$. 
The purpose of this paper is the study of the branching laws between $G$ and $H$.

Let $T$ be a maximal torus of $G$ that we choose $\theta$-invariant. Let $\tgot$ be the Lie algebra of $T$.
Let $\Lambda\subset\tgot^*$ be the lattice of weights, and let $\tgot_{+}^*$ 
be a Weyl chamber. The irreducible representations of $G$ are parametrized by the semi-group $\Lambda_+:=\Lambda\cap\tgot^*_{+}$ 
of dominant weights. 

Let $\lambda\in \Lambda_+$. In order to study the restriction $V_{\lambda}^{G}\vert_{H}$ of the irreducible $G$-representation $V_{\lambda}^{G}$, we consider the $H$-action on the flag variety $\Fcal=G/T$ of $G$. An important object is the $H$-invariant subset 
$$
Z_\theta\subset \Fcal
$$ 
formed of the elements $x\in \Fcal$ for which the stabilizer subgroup $G_x:=\{g\in G, gx=x\}$  is stable under $\theta$. In other words,
$gT\in Z_\theta$ if and only if $g^{-1}\theta(g)$ belongs to the normalizer subgroup $N(T)$.  A well-known result tells us that the group $H$ has finitely many 
orbits in $Z_\theta$, and that the map $\Ocal\in H_\Cbb\backslash\Fcal\longmapsto \Ocal\cap Z_\theta \in H\backslash Z_\theta$ 
is bijective \cite{Matsuki79,Rossmann79,Richardson-Springer90,Mi-Vi-Uz}.

Let $x\in Z_\theta$.  The stabilizer subgroup $G_{x}$ is a maximal torus in $G$, stable under $\theta$, with Lie algebra $\ggot_{x}$. 
We will also consider the abelian subgroup $H_{x}:=G_x\cap H$ (that is not necessarily connected). Any weight  $\mu\in\Lambda$ 
determines a character $\Cbb_{\mu_x}$ of the torus $G_{x}$ by taking $\mu_x=g\cdot\mu$ if $x=gT\in \Fcal$.

We denote by $\Rgot_x\subset \ggot^*_x$ the set of roots relative to the action of the Cartan subalgebra $\ggot_x$ on $\ggot\otimes\Cbb$. The map 
$\mu\in\Rgot\mapsto \mu_x\in\Rgot_x$ is an isomorphism, and we take $\Rgot_x^+\subset\Rgot_x$ as the image of $\Rgot^+\subset\Rgot$ through this isomorphism.

The involution $\theta$ leaves the set $\Rgot_x$ invariant, and $\alpha \in\Rgot_x$ is an {\em imaginary root} if $\theta(\alpha)=\alpha$. 
If $\alpha$ is imaginary, the subspace $(\ggot\otimes\Cbb)_\alpha$ is $\theta$-stable. There are two cases. If the action of $\theta$ on  
$(\ggot\otimes\Cbb)_\alpha$ is trivial then $\alpha$ is {\em compact imaginary}.  If the action of $-\theta$ on  
$(\ggot\otimes\Cbb)_\alpha$ is trivial, then $\alpha$ is {\em non-compact imaginary}. We denote respectively by $\Rgot_x^{\hbox{\scriptsize ci}}$ and 
by $\Rgot_x^{\hbox{\scriptsize nci}}$ the subsets of compact imaginary and non-compact imaginary roots, and we introduce the following $G_x$-modules
$$
\mathbb{E}^{\hbox{\scriptsize ci}}_{x}:=\sum_{\alpha\in\Rgot^{^{\rm ci}}_x\cap\Rgot^{+}_{x}}(\ggot\otimes\Cbb)_\alpha,
\quad
\quad
\mathbb{E}^{\hbox{\scriptsize nci}}_{x}:=\sum_{\alpha\in\Rgot_x^{^{\rm nci}}\cap\Rgot^{+}_{x}}(\ggot\otimes\Cbb)_\alpha.
$$

The weight 
$$
\delta(x):=\frac{1}{2}\sum_{\stackrel{\alpha\in\Rgot_{x}^{+}\cap\theta(\Rgot_{x}^{+})}{\theta(\alpha)\neq\alpha}}\alpha 
$$ 
defines a character $\Cbb_{\delta(x)}$ of the abelian group $H_x$. 

We denote by $R(H)$ and by $R(H_{x})$ the representations rings of the compact Lie groups $H$ and $H_{x}$. An element $E\in R(H)$ can be represented
as a {\em finite} sum $E=\sum_{V\in\widehat{H}}m_V V$, with $m_V\in\Zbb$. We denote by $\widehat{R}(H)$ (resp. $\widehat{R}(H_{x})$) the space 
of $\Zbb$-valued functions on $\widehat{H}$ (resp. $\widehat{H_{x}}$). An element $E\in \widehat{R}(H)$ can be represented as an {\em infinite} sum 
$\sum_{V\in\widehat{H}}m_V V$, with $m_{V}\in\Zbb$. The induction map ${\rm Ind}^{H}_{H_{x}}: \widehat{R}(H_{x})\to \widehat{R}(H)$ is the dual of the restriction morphism $R(H)\to R(H_{x})$.

Let $m_x= \frac{1}{2}| \Rgot_{x}^{+}\cap\theta(\Rgot_{x}^{+})\cap\{\theta(\alpha)\neq\alpha\}| + \dim \mathbb{E}^{\hbox{\scriptsize\rm nci}}_{x}\in\mathbb{N}$.

\medskip

The main result of this paper is the following theorem.

\begin{theo}\label{theo-principal}
Let $\lambda\in\Lambda_+$. We have the decomposition 
\begin{equation}\label{eq:theoreme-principal}
V_{\lambda}^{G}\vert_{H}=\sum_{Hx\,\in\, H\backslash Z_{\theta}} Q_{Hx}(\lambda)
\end{equation}
 where the terms $Q_{Hx}(\lambda)\in \widehat{R}(H)$ are defined by the following relation :
$$
Q_{Hx}(\lambda)=(-1)^{m_{x}} {\rm Ind}^{H}_{H_{x}}\left(\Cbb_{\lambda_x+\delta(x)}\otimes 
\det(\mathbb{E}^{\hbox{\scriptsize\rm nci}}_{x})\otimes {\rm Sym}(\mathbb{E}^{\hbox{\scriptsize\rm nci}}_{x})\otimes\bigwedge 
\mathbb{E}^{\hbox{\scriptsize\rm  ci}}_{x}\right).
$$
Here ${\rm Sym}(\mathbb{E}^{\hbox{\scriptsize\rm nci}}_{x})$, which is the symmetric algebra of 
$\mathbb{E}^{\hbox{\scriptsize\rm nci}}_{x}$, is an admissible representation of $H_x$ and $\bigwedge \mathbb{E}^{\hbox{\scriptsize\rm  ci}}_{x}=
\bigwedge^{+}\mathbb{E}^{\hbox{\scriptsize\rm  ci}}_{x}\ominus\bigwedge^{-} \mathbb{E}^{\hbox{\scriptsize\rm  ci}}_{x}$ is a virtual representation of $H_{x}$.
\end{theo}

We give now another formulation for decomposition (\ref{eq:theoreme-principal}) using the (right) action of the Weyl group $W=N(T)/T$ on the flag variety $\Fcal$. If 
$x=gT\in \Fcal$ and $w\in W$ we take $xw:=gwT$. We notice that $Z_\theta$ is stable under the action of $W$ and that the quotient $Z_\theta/W$ parametrizes the set 
of maximal tori of $G$ stable under $\theta$.

We associate to an element $x=gT\in Z_\theta$ the subgroup $W^H_x\subset W$ defined by the relation $w\in W^H_x\Longleftrightarrow Hxw=Hx$. 
We denote by $H\backslash Z_{\theta}\slash W$ the quotient of $Z_\theta$ by the action of $H\times W$, and by $\bar{x}\in H\backslash Z_{\theta}\slash W$ the 
image of $x\in Z_\theta$ through the quotient map. We associate to $\bar{x}\in H\backslash Z_{\theta}\slash W$ the element $Q_{\bar{x}}(\lambda)\in \widehat{R}(H)$ defined as follows
$$
Q_{\bar{x}}(\lambda)=\sum_{\bar{w}\in W^H_x\backslash W} Q_{Hxw}(\lambda).
$$
The previous theorem says then that $V_{\lambda}^{G}\vert_{H}=\sum_{\bar{x}\,\in\, H\backslash Z_{\theta}\slash W}Q_{\bar{x}}(\lambda)$. Here is a new formulation of Theorem \ref{theo-principal}.

\begin{theo}\label{theo-principal-bis}
We have  $V_{\lambda}^{G}\vert_{H}=\sum_{\bar{x}\,\in\, H\backslash Z_{\theta}\slash W}Q_{\bar{x}}(\lambda)$
where $Q_{\bar{x}}(\lambda)\in \wR(H)$ has the following description
$$
Q_{\bar{x}}(\lambda)=
{\rm Ind}^{H}_{H_{x}}\left(\mathbb{M}_x(\lambda)\otimes \Cbb_{\delta(x)}\otimes\bigwedge \mathbb{E}^{\hbox{\scriptsize\rm  ci}}_{x}\right),
$$
for some\footnote{The precise expression of  $\mathbb{M}_x(\lambda)$ is given in Proposition \ref{prop:formule-M-lambda}.} $\mathbb{M}_x(\lambda)\in \wR(H_x)$.
\end{theo}

\bigskip

We finish this section by giving two basic examples associated to the group $SU(2)$. Here the flag variety of $SU(2)$ is the $2$-dimensional sphere $\mathbb{S}^2$. 
For $n\geq 0$, we denote by $V_n$ the irreducible representation of $SU(2)$ of dimension $n+1$. 


\medskip

{\bf Example 1}. $G=SU(2)$ and the involution $\theta$ is the conjugation by the matrix $\left(\begin{array}{cc} 1& 0 \\ 0 & -1\end{array}\right)$. The subgroup fixed by $\theta$ is the torus $T\simeq U(1)$ and the 
critical set $Z_\theta\subset\mathbb{S}^2$ is composed of the poles $S,N$ and the equator $E$, so that $T\backslash Z_\theta$ has three terms. We take $\lambda=n$ in $\widehat{SU(2)}\simeq \mathbb{N}$.

For $Hx=E$, we have  $\mathbb{E}^{\hbox{\scriptsize\rm nci}}_{x}=\mathbb{E}^{\hbox{\scriptsize\rm ci}}_{x}=\{0\}$, $H_x\simeq \mathbb{Z}_2$, and 
$\Cbb_{\lambda_x+\delta(x)}=\Cbb_n\vert_{\mathbb{Z}_2}$.  The contribution of $E$ is then 
${\rm Ind}^{U(1)}_{\mathbb{Z}_2}(\Cbb_n\vert_{\mathbb{Z}_2})=\Cbb_{n}\otimes\sum_{k\in\mathbb{Z}} \Cbb_{2k}$. 

For $Hx=N$, we have $H_x=T$, $\mathbb{E}^{\hbox{\scriptsize\rm nci}}_{x}=\Cbb_2$, $\mathbb{E}^{\hbox{\scriptsize\rm ci}}_{x}=\{0\}$, and 
$\Cbb_{\lambda_x+\delta(x)}=\Cbb_{n}$. The contribution of $N$ is then $-\Cbb_{n+2}\otimes{\rm Sym}(\Cbb_{2})$. 

For $Hx=S$, we have $H_x=T$, $\mathbb{E}^{\hbox{\scriptsize\rm nci}}_{x}=\Cbb_{-2}$, $\mathbb{E}^{\hbox{\scriptsize\rm ci}}_{x}=\{0\}$, and 
$\Cbb_{\lambda_x+\delta(x)}=\Cbb_{-n}$. The contribution of $S$ is then $-\Cbb_{-n-2}\otimes{\rm Sym}(\Cbb_{-2})$. 

Finally, Relations (\ref{eq:theoreme-principal}) become 
\begin{eqnarray*}
V_n\vert_T&=& \Cbb_{n}\otimes\sum_{k\in\mathbb{Z}} \Cbb_{2k}  - \Cbb_{-n-2}\otimes{\rm Sym}(\Cbb_{-2}) - \Cbb_{n+2}\otimes{\rm Sym}(\Cbb_{2})\\
&=&\sum_{k=-n}^{0} \Cbb_{2k+n}.
\end{eqnarray*}

\medskip

{\bf Example 2}. $G=SU(2)\times SU(2)$ and the involution $\theta$ is the map $(a,b)\mapsto (b,a)$. The subgroup fixed by $\theta$ is $SU(2)$ embedded diagonally and the critical set $Z_\theta\subset\mathbb{S}^2\times \mathbb{S}^2$ is equal to 
the union of the orbits $SU(2)\cdot(N,N)$ and $SU(2)\cdot(S,N)$. Let $\lambda=(n,m)\in \widehat{G}$.

For $x=(N,N)$ or $x=(S,N)$ we have $\mathbb{E}^{\hbox{\scriptsize\rm nci}}_{x}=\mathbb{E}^{\hbox{\scriptsize\rm ci}}_{x}=\{0\}$ and $H_x\simeq T$. For 
$x=(N,N)$ we have $\lambda_x+\delta(x)=m+n+2$, and for $x=(S,N)$ we have $\lambda_x+\delta(x)=m-n$. Relations (\ref{eq:theoreme-principal}) give then 
$$
V_n\otimes V_m=   \ind_{T}^{SU(2)}(\Cbb_{m-n}) - \ind_{T}^{SU(2)}(\Cbb_{m+n+2}).
$$
It is not difficult to see that the previous identities correspond to the classical Clebsch-Gordan relations (see Example \ref{SU(2)-Clebsch-Gordan}).

\medskip

Here is a brief overview of the article. Sections 2 and 3 are devoted to the proof of our main result. In Section 4, we detail the case of $U(p)\times U(q)\subset U(n)$: in particular, we explain the branching formula we obtain for the restriction of $U(n)$ to $U(n-1)$. In the last section, we recall Kostant's branching formula and explain the formula it gives in the case of the restriction of $U(n) $ to $U(n-1)$, in order to compare it with our own formula.

\begin{center}
\bf Notations
\end{center}

Throughout the paper :
\begin{itemize}
\item $G$ denotes a compact connected Lie group with Lie algebra $\ggot$.
\item $T$ is a maximal torus in $G$ with Lie algebra $\tgot$.
\item $\Lambda\subset \tgot^*$ is the weight lattice  of $T$ : every $\mu\in \Lambda$ defines a $1$-dimensional
$T$-representation, denoted by $\Cbb_\mu$, where $t=\exp(X)$ acts by $t^\mu:= e^{i\langle\mu, X\rangle}$.
\item The coadjoint action of $g\in G$ on $\xi\in \ggot^*$ is denoted by $g\cdot\xi$.
\item When a Lie group $K$ acts on set $X$, the stabilizer subgroup of $x\in X$ is denoted by $K_x:=\{k\in K\ \vert\ k\cdot x=x\}$ and the Lie algebra of $K_x$ is denoted by $\kgot_x$.
\item  When a Lie group $K$ acts on a manifold $M$, we denote by $X\cdot m:=\frac{d}{dt} e^{tX}\cdot m\vert_{t=0}$, $m\in M$, the infinitesimal action of $X\in \kgot$ 
on $M$. 
\end{itemize}

\medskip 

{\bf Acknowledgments.} We would like to thank the referees for their invaluable advice, which enabled me to improve this text.

 \section{Non abelian localization}

Our main result is obtained by means of a non-abelian localization of the Riemann-Roch character on the flag variety $\Fcal$ of $G$. 
For that purpose we will use the family $(\Omega_r)_r$ of symplectic structure parametrized by the interior of the Weyl chamber $\tgot^*_+$. 
The symplectic structure $\Omega_r$ comes from the identication $gT\to g\cdot r$ of $\Fcal$ with the coadjoint orbit $Gr$. 
The moment map $\Phi_r:\Fcal\to\ggot^*$ associated to the action of $G$ on $(\Fcal,\Omega_r)$ is the map 
$gT \mapsto g\cdot r$.

At the level of Lie algebras we have $\ggot=\hgot\oplus\qgot$ where 
$\hgot=\ggot^\theta$ and $\qgot=\ggot^{-\theta}$. For any $\xi\in \ggot=\hgot\oplus\qgot$, 
we denote by $\xi^+$ his $\hgot$-part and by $\xi^-$ his $\qgot$-part. We use a $G$-invariant scalar product $(-,-)$ on $\ggot$ such that 
the involution $\theta$ is an orthogonal map. It induces identifications $\ggot^*\simeq\ggot$,  $\hgot^*\simeq\hgot$ and $\qgot^*\simeq\qgot$. 

The moment map $\Phi^H_r:\Fcal\to\hgot^*$ associated to the action of $H$ on $(\Fcal,\Omega_r)$ is the map 
$gT \mapsto (g\cdot r)^+$.

\subsection{Matsuki duality}
Consider the complex reductive groups $G_\Cbb$ and $H_\Cbb$ associated to the compact Lie groups $G$ and $H$.  
Let $L\subset G_\Cbb$ be the real form such that $H\subset L$ is a maximal compact subgroup of $L$. 

Matsuki duality is the statement that a one-to-one correspondence exists between the $H_\Cbb$-orbits and the $L$-orbits in $\Fcal$; two orbits are in duality 
when their intersection is a single orbit of $H$.

Uzawa, and Mirkovic-Uzawa-Vilonen  \cite{Uzawa90,Mi-Vi-Uz} proved the Matsuki correspondence by showing that both $H_\Cbb$-orbits and $L$-orbits in 
$\Fcal$ are parametrized by the $H$-orbits in the set of critical points of the function $\|\Phi^H_r\|^2: \Fcal\to \Rbb$.

First we recall the elementary but fundamental fact that the subset $Z_\theta$ is equal to the set of critical points of the function $\|\Phi^H_r\|^2$ 
\cite{Mi-Vi-Uz,Bremigan-Lorch}.

\begin{lem}\label{lem-Uzawa}
Let $x=gT\in \Fcal$ and $r\in {\rm Interior}(\tgot^*_+)$. The following statements are equivalent:
\begin{itemize}
\item[i)] the subalgebra $\ggot_x$ is invariant under $\theta$ (i.e. $x\in Z_\theta$),
\item[ii)] $g^{-1}\theta(g)\in N(T)$,
\item[iii)] $x$ is a critical point of the function $\|\Phi^H_r\|^2$,
\item[iv)] $(g\cdot r)^+$ and $(g\cdot r)^-$ commute.
\end{itemize}
\end{lem}

\begin{proof} Let $n_g= g^{-1}\theta(g)$ and let $r$ be a regular element of $\tgot^*\simeq \tgot$. Since $\ggot_x=Ad(g)\tgot$ we see that 
\begin{eqnarray*}
 \theta(\ggot_x)=\ggot_x  &\Longleftrightarrow &n_g\in N_G(T)\\
 &\Longleftrightarrow & [n_g\cdot\theta(r),r]=0\\
 &\Longleftrightarrow & [\theta(g\cdot r),g\cdot r]=0\\
 &\Longleftrightarrow & [(g\cdot r)^+,(g\cdot r)^-]=0.
\end{eqnarray*}

A small computation shows that for any $X\in\ggot$ the derivative of the function $t\mapsto \|\Phi^H_r(e^{tX}x)\|^2$ at $t=0$ 
is equal to $(X,[g\cdot r,\theta(g\cdot r)])$. Hence $x=gT$ is a critical point of the function $\|\Phi^H_r\|^2$ if and only if $[g\cdot r,\theta(g\cdot r)]=0$. 
Finally we have proved that the statements $i),ii),iii)$ and $iv)$ are equivalent. 
\end{proof}

Let us check the other easy fact.

\begin{lem}
The set $H\backslash Z_\theta$ is finite.
\end{lem}

\begin{proof}
Let $x=gT\in Z_\theta$. A neighborhood of $x$ is defined by elements of the form $e^{X}e^{Y}x$ where $X\in\hgot$ and $Y\in\qgot$. Now we see that 
$e^{X}e^{Y}gT\in Z_\theta$ if and only if $e^{-2g^{-1}Y}\in N(T)$. If $Y$ is sufficiently small the former relation is equivalent to $g^{-1}Y\in \tgot$, 
and in this case $e^{X}e^{Y}x=e^{X}x$. We have proved that any element in $H\backslash Z_\theta$ is isolated. As $H\backslash Z_\theta$ is compact, we can conclude that 
$H\backslash Z_\theta$ is finite.  
\end{proof}

\subsection{Borel-Weil-Bott theorem}

We first recall the Borel-Weil-Bott theorem. The flag manifold $\Fcal$ is equipped with the $G$-invariant complex structure such that 
$$
\T_{eT}\Fcal\simeq \sum_{\alpha\in \Rgot^+}(\ggot\otimes\Cbb)_\alpha
$$
is an identity of $T$-modules. Let us consider the tangent bundle $\T\Fcal$ as a complex vector bundle on $\Fcal$ with the invariant  
Hermitian structure $h_\Fcal$ induced by the invariant scalar product on $\ggot$.

Any weight $\lambda\in \Lambda$ defines a line bundle $\Lcal_\lambda\simeq G\times_T\Cbb_\lambda$ on $\Fcal$.

\begin{defi}
 We associate to a weight $\lambda\in \Lambda$ 
 
 $\bullet$ the spin-c bundle on $\Fcal$ 
$$
\Scal_\lambda:= \bigwedge_{\Cbb}\T \Fcal\otimes \Lcal_\lambda,
$$

$\bullet$ the Riemann-Roch character $\RR_G(\Fcal,\Lcal_\lambda)\in R(G)$ which is the equivariant index of the 
Dirac operator $D_\lambda$ associated to the spin-c structure $\Scal_\lambda$.
\end{defi}

The Borel-Weil-Bott theorem asserts that $V_\lambda^G=\RR_G(\Fcal,\Lcal_\lambda)$ when $\lambda$ is dominant. 
Now we consider the restriction $V_\lambda^G\vert_H=\RR_H(\Fcal,\Lcal_\lambda)$. In the next section, we will explain how we can localize the $H$-equivariant Riemann-Roch character $\RR_H(\Fcal,\Lcal_\lambda)$ on the critical set of the function $\|\Phi^H_r\|^2$ \cite{pep-RR}.

\subsection{Localization of the Riemann-Roch character}\label{sec:localisation}
In this section, we recall how we perform the ``Witten  non-abelian localization'' of the Riemann-Roch character with the help of the moment map $\Phi^H_r:\Fcal\to\hgot^*$ attached to a  regular element $r$ of the Weyl chamber  \cite{pep-RR,Ma-Zhang14,pep-vergne:witten}. 
 
Let us denote by $X\mapsto [X]_{\ggot/\tgot}$ the projection $\ggot\to\ggot/\tgot$. The Kirwan vector field $\kappa_r$ on $\Fcal$ is defined as follows:
$$
\kappa_r(x)=-\Phi^H_r(x)\cdot x \in \T_x\Fcal.
$$
Through the identification $\ggot/\tgot\simeq\T_x\Fcal, X\mapsto \frac{d}{dt}\vert_{t=0} ge^{tX}T$, the vector $\kappa_r(x)\in\T_x\Fcal$ is equal to 
$[g^{-1}\theta(g)\cdot r]_{\ggot/\tgot}$. Hence the set $Z_\theta\subset \Fcal$ is exactly the set where $\kappa_r$ vanishes.

Let $D_0$ be the Dirac operator associated to the spin-c structure $\Scal_0= \bigwedge_{\Cbb}\T \Fcal$. The principal symbol of the elliptic 
operator $D_0$ is the bundle map \break $\sigma(\Fcal)\in \Gamma(\T^* \Fcal, \hom(\bigwedge^+_\Cbb\T\Fcal,\bigwedge^-_\Cbb \T\Fcal))$ 
defined by the Clifford action 
$$
\sigma(\Fcal)(x,\nu)=\clif_{x}(\tilde{\nu}): {\bigwedge}^+_\Cbb\T_x\Fcal\to {\bigwedge}^-_\Cbb\T_x\Fcal.
$$
Here $\nu\in \T^*_x \Ocal\simeq \tilde{\nu}\in \T_x \Ocal$ is the one to one map associated to the identification
$\ggot^*\simeq \ggot$ (see \cite{B-G-V}).

Now we deform the elliptic symbol $\sigma(\Fcal)$ by means of the vector field $\kappa_r$ \cite{pep-RR,pep-vergne:witten}.

\begin{defi}\label{def:pushed-sigma}
The symbol  $\sigma(\Fcal)$ shifted by the vector field $\kappa_r$ is the
symbol on $\Fcal$ defined by
$$
\sigma_r(\Fcal)(x,\nu)=\clif_{x}(\tilde{\nu}-\kappa_r(x))
$$
for any $(x,\nu)\in\T^* \Fcal$.
\end{defi}

Consider an $H$-invariant open subset $\Ucal\subset \Fcal$ such that $\Ucal\cap Z_\theta$ 
is compact in $\Fcal$. Then the restriction $\sigma_r(\Fcal)\vert_\Ucal$ is a $H$-transversally elliptic symbol on $\Ucal$, 
and so its equivariant index is a well defined element in $\wR(H)$ (see \cite{Atiyah74,pep-RR,pep-vergne:witten}).

Thus we can define the following localized equivariant indices.

\begin{defi}\label{def:indice-localise}
Let $Hx\subset Z_\theta$. We denote by
$$
\RR_H(\Fcal,\Lcal_\lambda,\Phi_r^H, Hx)\ \in\ \widehat{R}(H)
$$
the equivariant index of $\sigma_r(\Fcal)\otimes\Lcal_\lambda\vert_\Ucal$ where 
$\Ucal$ is an invariant neighbourhood of $Hx$ so that $\Ucal\cap Z_\theta=Hx$.
\end{defi}

We proved in \cite{pep-RR} that the following decomposition holds in $\wR(H)$:
$$
\RR_H(\Fcal,\Lcal_\lambda)=\sum_{Hx\in H\backslash Z_\theta}
\RR_H(\Fcal,\Lcal_\lambda, \Phi_r^H, Hx).
$$

The computation of the characters $\RR_H(\Fcal,\Lcal_\lambda, \Phi_r^H, Hx)$ will be handle in Section \ref{sec:computation-Q-Hx-lambda}. To undertake these calculations we 
need to describe geometrically  a neighborhood of $Hx$ in $\Fcal$. This is the goal of the next section.

\subsection{Local model near $Hx\subset Z_\theta$}

Let $x=gT\in Z_\theta$. We need to compute a symplectic model of a neighborhood of $Hx$ in $(\Fcal,\Omega_r)$. Here we use the identification
$\ggot\simeq \ggot^*$ given by the choice of an invariant scalar product. Let $\mu=g\cdot r$ that we write $\mu=\mu^+ +\mu^-$ where 
$\mu^+\in \hgot$ and $\mu^-\in \qgot$.

The tangent space $\T_x \Fcal$ is equipped with the symplectic two form 
$\Omega_r\vert_x$:
$$
\Omega_r\vert_x(X\cdot x,Y\cdot x)=(\mu,[X,Y]),\quad X,Y\in \ggot.
$$

We need to understand the structure of the symplectic vector space $(\T_x\Fcal,\Omega_r\vert_x)$. 
If $\agot\subset \ggot$ is a vector subspace we denote by $\agot\cdot x:=\{X\cdot x, X\in\agot\}$ the corresponding subspace of $\T_x\Fcal$. The symplectic orthogonal of $\agot\cdot x$ is denoted by $(\agot\cdot x)^{\perp,\Omega}$. 

If $\agot,\bgot$ are two subspaces, a small computation gives that
\begin{equation}\label{eq:orthogonal-a-x}
(\agot\cdot x)^{\perp,\Omega}\cap \bgot\cdot x \simeq \agot^\perp\cap [\bgot,\mu],
\end{equation}
where $\agot^\perp\subset \ggot$ is the orthogonal of $\agot$ relatively to the scalar product.

We denote by $\ggot_{\mu^+}=\hgot_{\mu^+}\oplus\qgot_{\mu^+}$ the subspaces fixed by $ad(\mu^+)$. Notice that 
$\ggot_{\mu}=\ggot_x$ is an abelian subalgebra containing $\mu^+$ since $[\mu^+,\mu^-]=0$. It follows that $\ggot_x\subset \ggot_{\mu^+}$. 

\begin{lem}$[\ggot,\mu^+]\cdot x$, $\ggot_{\mu^+}\cdot x$ and $[\hgot,\mu^+]\cdot x$ are symplectic subspaces of $\T_x\Fcal$.
\end{lem}
\begin{proof} 
It is a direct consequence of (\ref{eq:orthogonal-a-x}).  
\end{proof}

\medskip

We consider now the symplectic subspace $V_x \subset \T_x\Fcal$ defined by the relation
\begin{equation}\label{eq:V-x}
V_x= ([\hgot,\mu^+]\cdot x)^{\perp,\Omega}\cap [\ggot,\mu^+]\cdot x.
\end{equation}
A small computation shows that $X\cdot x\in V_x$ if and only if $[X,\mu]\in [\qgot,\mu^+]$.

We have the following important Lemma. 

\begin{lem}\label{lem:decomposition-T-x-O}
\begin{itemize}
\item We have the following decomposition
\begin{equation}\label{eq:T-x-F}
\T_x\Fcal    = \ggot_{\mu^+}\cdot x \stackrel{\perp}{\oplus} \left[\hgot,\mu^+\right]\cdot x\stackrel{\perp}{\oplus} V_x
\end{equation}
where $\perp$ stands for the orthogonal relative to $\Omega_r\vert_x$. 
\item $\ggot_{\mu^+}\cdot x$  is  symplectomorphic to $\hgot_{\mu^+}/\hgot_x\oplus (\hgot_{\mu^+}/\hgot_x)^*$.
\item $\left[\hgot,\mu^+\right]\cdot x$  is symplectomorphic to $\hgot/\hgot_{\mu^+}$ equipped 
with the symplectic structure $\Omega_{\mu^+}(\bar{u},\bar{v})=(\mu^+,[u,v])$.
\item $V_x$  is symplectomorphic to $(\hgot\cdot x)^{\perp,\Omega}\slash \left[(\hgot\cdot x)^{\perp,\Omega}\cap \hgot\cdot x\right]$.
\end{itemize}
\end{lem}

\begin{proof}
If we use the decomposition $\ggot= \ggot_{\mu^+}\oplus[\ggot,\mu^+]$ and the fact that the abelian subalgebra $\ggot_x$ is contained in $\ggot_{\mu^+}$ we obtain 
$$
\T_x\Fcal=  \ggot_{\mu^+}\cdot x \oplus [\ggot,\mu^+]\cdot x.
$$
It is obvious to check that the subspaces $[\ggot,\mu^+]\cdot x$ and $\ggot_{\mu^+}\cdot x$ are orthogonal relatively to the symplectic form 
$\Omega_r\vert_x$. Since $[\hgot,\mu^+]\cdot x$ is a symplectic subspace we have $[\ggot,\mu^+]\cdot x=[\hgot,\mu^+]\cdot x \stackrel{\perp}{\oplus} V_x$ where 
$V_x$ is defined by (\ref{eq:V-x}). The first point is proved.

The identities $\ggot_x=\theta(\ggot_x)=\ggot_{\theta(x)}$ imply the decompositions $\ggot_x=\hgot_x\oplus\qgot_x$ and $\ggot_{\mu^+}\cdot x=\qgot_{\mu^+}\cdot x\oplus \hgot_{\mu^+}\cdot x$. The vector subspace $\hgot_{\mu^+}\cdot x$ is isomorphic to $\hgot_{\mu^+}/\hgot_x$, and the map $v\mapsto \Omega_r\vert_x(v,-)$ defines an isomorphism between $\qgot_{\mu^+}\cdot x$ and the dual of $\hgot_{\mu^+}\cdot x$. 
The second point is proved. 

For the third point we use the isomophism $j:[\hgot,\mu^+]\to \hgot/\hgot_{\mu^+}$ induced by the projection $\hgot \to \hgot/\hgot_{\mu^+}$. Then the map 
$\bar{u}\mapsto j(\bar{u})\cdot x$ defines a symplectomorphism between $(\hgot/\hgot_{\mu^+},\Omega_{\mu^+})$ and $[\hgot,\mu^+]\cdot x$.

Now we see that (\ref{eq:T-x-F}) together with the decomposition $\hgot\cdot x=$\break  $\hgot_{\mu^+}\cdot x \stackrel{\perp}{\oplus} [\hgot,\mu^+]\cdot x$   leads to 
\begin{eqnarray*}
(\hgot\cdot x)^{\perp,\Omega}&=& ([\hgot,\mu^+]\cdot x)^{\perp,\Omega}\cap(\hgot_{\mu^+}\cdot x)^{\perp,\Omega}\\
&=&\hgot_{\mu^+}\cdot x \oplus V_x\\
&=&\left[(\hgot\cdot x)^{\perp,\Omega}\cap \hgot\cdot x\right]\oplus V_x.
\end{eqnarray*}
The last point follows.  
\end{proof}

\medskip

We denote by $\Omega_{V_x}$ the restriction of $\Omega_r\vert_x$ on the symplectic vector subspace $V_x$. The action of $H_x$ on $(V_x,\Omega_{V_x})$ 
is Hamiltonian, with moment map $\Phi_{V_x}: V_x\to \hgot_x^*$ defined by the relation
$$
\langle\Phi_{V_x}(v), A\rangle= \frac{1}{2}\Omega_{V_x}(v,Av),\quad v\in V_x,\ A\in \hgot_x.
$$

Thanks to Lemma \ref{lem:decomposition-T-x-O}, we know that the $H_x$-symplectic vector space $(\T_x\Fcal,\Omega_r\vert_x)$ admits the following decomposition
$$
\T_x\Fcal \simeq \hgot_{\mu^+}/\hgot_x\oplus (\hgot_{\mu^+}/\hgot_x)^* 
\stackrel{\perp}{\oplus}\hgot/\hgot_{\mu^+}
\stackrel{\perp}{\oplus} V_x.
$$

Thanks to the normal form Theorem of Marle \cite{Marle85} and Guillemin-Sternberg \cite{Guillemin-Sternberg84}, we get the following result.

\begin{coro}\label{coro:model-symplectic}
An $H$-equivariant symplectic model of a neighborhoood of $Hx$ in $\Fcal$ is 
$\Fcal_x:= H\times_{H_{\mu^+}} Y_x$
where  
$$
Y_x= H_{\mu^+}\times_{H_x}\left((\hgot_{\mu^+}/\hgot_x)^*\times V_x\right).
$$
The corresponding moment map on $\Fcal_x$ is
$$
\Phi_{\Fcal_x}([h; \eta,v])= h(\eta + \mu^+ + \Phi_{V_x}(v))
$$
for $[h; \eta,v]\in H\times_{H_x}\left((\hgot_{\mu^+}/\hgot_x)^*\times V_x\right)$.
\end{coro}

\bigskip

We finish this section by computing a compatible complex structure on $V_x$. 

By definition, the map that sends $X\cdot x$ to $[X,\mu]$ defines an isomorphism $i:V_x\to [\qgot,\mu^+]$. 
The adjoint map $ad(\mu)$ defines also an automorphism of $[\ggot,\mu^+]$: for any $X\in[\ggot,\mu^+]$ we denote by 
$\tilde{X}\in [\ggot,\mu^+]$ the unique element such that $ad(\mu)\tilde{X}=X$.

The symplectic structure $\Omega_\mu:=(i^{-1})^*\Omega_{V_x}$ satisfies the relations
$$
\Omega_\mu(X,Y)=(\mu,[\tilde{X},\tilde{Y}])=(X,\tilde{Y})=-(\tilde{X},Y),\quad \forall X,Y\in [\qgot,\mu^+].
$$

We work with the following $H_x$-equivariant maps
\begin{itemize}
    \item the one to one map $T_\mu:=-ad(\mu)ad(\theta(\mu)):[\ggot,{\mu^+}]\to [\ggot,{\mu^+}]$,
    \item the complex structure $J_{\mu^+}= ad(\mu^+)(-ad(\mu^+)^2)^{-1/2}$ on $[\ggot,{\mu^+}]$.
\end{itemize}
The map $T_\mu$ restricts to a one to one map 
$T_x:[\qgot,{\mu^+}]\to [\qgot,{\mu^+}]$ and $J_{\mu^+}$ defines a complex structure on $[\qgot,{\mu^+}]$ (still denoted by $J_{\mu^+}$). 

Let $S_x:=(T_x^2)^{-1/2} T_x$. The map $J_{V_x}:= J_{\mu^+}\circ S_x$ defines a $H_x$-invariant complex structure on $[\qgot,{\mu^+}]$.

\medskip

\begin{lem}\label{lem:V-x-complexe}The $H_x$-symplectic space $(V_x,\Omega_{V_x})$ is isomorphic to $[\qgot,{\mu^+}]$ equipped with the symplectic form
$\Omega_{\mu}^1(v,w)= (J_{V_x}v,w)$.
\end{lem}

\begin{proof} We know already that $(V_x,\Omega_{V_x})\simeq([\qgot,{\mu^+}],\Omega_\mu)$. If one takes $L=$ \break $T_x\circ (-ad(\mu^+)^2)^{-1/4}\circ (T_x^2)^{-1/4}$, 
we check easily that $\Omega_{\mu}(L(v),L(w))=(J_{V_x}v,w)$.  
\end{proof}

 \section{Proof of the main theorem}
 
 We start with the following lemma.

\begin{lem}
The quantity $\RR_H(\Fcal,\Lcal_\lambda,\Phi_r^H, Hx)$ does not depend on the choice of the regular element $r$ in the Weyl chamber. 
In the following we will denote it by $Q_{Hx}(\lambda)\in\wR(H)$.
\end{lem}

\begin{proof}
Let $r_0,r_1$ be two regular elements of the Weyl chamber. For $t\in [0,1]$, we consider the regular element $r(t)=tr_1+(1-t)r_0$: the Kirwan vector field 
$\kappa_{r(t)}$ vanishes exactly on $Z_\theta$ for any $t\in [0,1]$.  If $\Ucal$ is an invariant neighbourhood of $Hx$ so that $\Ucal\cap Z_\theta=Hx$, then 
$t\in [0,1]\mapsto \sigma_{r(t)}(\Fcal)\otimes\Lcal_\lambda\vert_\Ucal$ defines an homotopy of transversally elliptic symbols. Accordingly, the equivariant index of 
$\sigma_{r_0}(\Fcal)\otimes\Lcal_\lambda\vert_\Ucal$ and $\sigma_{r_1}(\Fcal)\otimes\Lcal_\lambda\vert_\Ucal$ are equal.  
\end{proof}

 \subsection{Computation of $Q_{Hx}(\lambda)$}\label{sec:computation-Q-Hx-lambda}

The computation of $Q_{Hx}(\lambda)$ is done in three steps.

 \subsubsection{Step 1: holomorphic induction}

Let $H_{\mu^+}\subset H$ be the stabilizer subgroup of $\mu^+:=\Phi^H_r(x)$. By Corollary \ref{coro:model-symplectic}, a symplectic $H$-equivariant  model of a neighborhoood of $Hx$ in 
$\Fcal$ is the manifold $H\times_{H_{\mu^+}} Y_x$ where  
$$
Y_x= H_{\mu^+}\times_{H_x}\left((\hgot_{\mu^+}/\hgot_x)^*\times V_x\right).
$$
The symplectic two form on $Y_x$ is built from the canonical symplectic structure on $H_{\mu^+}\times_{H_x}(\hgot_{\mu^+}/\hgot_x)^*\simeq \T^*
(H_{\mu^+}\slash {H_x})$ and the symplectic structure on $V_x$. The moment map relative to the action of $H_{\mu^+}$ on 
$Y_x$ is
$$
\Phi_{Y_x}([h; \eta,v])= h(\eta + \mu^+ + \Phi_{V_x}(v))\in \hgot^*_{\mu^+},
$$
for $[h; \eta,v]\in H_{\mu^+}\times_{H_x}\left((\hgot_{\mu^+}/\hgot_x)^*\times V_x\right)$. 

Let $\kappa_{Y_x}$ the Kirwan vector field on $Y_x$. It is immediate to check that $[h; \eta,v]\in \{\kappa_{Y_x}=0\}$ if and only if $\eta=0$ and 
$(\mu^+ + \Phi_{V_x}(v))\cdot v=0$. The map $v\in V_x\mapsto \mu^+\cdot v\in V_x$ is bijective and $v\mapsto \Phi_{V_x}(v)\cdot v$ is homogeneous of degree equal to $3$. Then 
there exists $\epsilon>0$ such that 
$$
(\mu^+ + \Phi_{V_x}(v))\cdot v=0 \quad {\rm and} \quad \|v\|\leq \epsilon \  \Longrightarrow \ v=0.
$$
In $Y_x$, we still denote by $x$ the point $[e,0,0]$. We equip $Y_x$ with an invariant almost complex structure that is compatible with the symplectic structure, and 
we denote by $\RR_{H_{\mu^+}}(Y_{x},\Lcal_\lambda\vert_{Y_x},\Phi_{Y_x}, H_{\mu^+} x)$ the Riemann-Roch character on $Y_x$ localized on the component
$H_{\mu^+} x\subset \{\kappa_{Y_x}=0\}$. 

The quotient $\hgot/\hgot_{\mu^+}$, which is equipped with the invariant complex structure  $J_{\mu^+}$ \break $:=ad(\mu^+)(-ad(\mu^+)^2)^{-1/2}$, is  a complex $H_{\mu^+}$-module.

In \cite{pep-RR}[Theorem 7.5], we proved that $Q_{Hx}(\lambda)=\RR_H(\Fcal,\Lcal_\lambda,\Phi_r^H, Hx)$ is equal to
\begin{equation}\label{eq:holomorphic induction}
{\rm Ind}_{H_{\mu^+}}^H\left( \RR_{H_{\mu^+}}(Y_{x},\Lcal_\lambda\vert_{Y_x},\Phi_{Y_x}, H_{\mu^+} x)\otimes \bigwedge \hgot/\hgot_{\mu^+}\right).
\end{equation}

 \subsubsection{Step 2: cotangent induction}

The map $\Phi_x(v):= \mu^+ + \Phi_{V_x}(v)$ is a moment map for the Hamiltonian action of $H_x$ on $V_x$. The moment map on the $H_{\mu^+}$-manifold 
$$
Y_x= H_{\mu^+}\times_{H_x}\left((\hgot_{\mu^+}/\hgot_x)^*\times V_x\right)
$$
is $\Phi_{Y_x}([h; \eta,v])= h(\eta +\Phi_x(v))\in \hgot^*_\mu$. 

Let $\kappa_{V_x}(v)=-\Phi_x(v)\cdot v$ be the Kirwan vector field on $V_x$. We are interested in the connected component $\{0\}$ of 
$\{\kappa_{V_x}=0\}$. We choose a compatible almost complex structure on the symplectic vector space $V_x$ and we denote by  
$\RR_{H_x}(V_x,\Phi_{x}, \{0\})\in \wR(H_x)$ the 
Riemann-Roch character localized on $\{0\}\subset \{\kappa_{V_x}=0\}$.

In Section 3.3 of \cite{pep-vergne:witten} we have proved that 
\begin{equation}\label{eq:cotangent induction}
\RR_{H_{\mu^+}}(Y_{x},\Lcal_\lambda\vert_{Y_x},\Phi_{Y_x}, H_{\mu^+} x)
={\rm Ind}_{H_x}^{H_{\mu^+}}\left(\RR_{H_x}(V_x,\Phi_{x}, \{0\}) \otimes \Lcal_\lambda\vert_{x}\right).
\end{equation}
 
 \subsubsection{Step 3: linear case}

We write $\qgot/\qgot_{\mu^+}$ for the vector space $[\qgot,\mu^+]$ equipped with the complex structure $J_{\mu^+}$. 
So $\qgot/\qgot_{\mu^+}$ is a $H_{\mu^+}$-module and we denote by ${\rm Sym}(\qgot/\qgot_{\mu^+})$ the corresponding symmetric algebra.

We need to compare the virtual $H_x$-modules $\bigwedge_{J_{V_x}} V_x$ and $\bigwedge_{-J_{\mu^+}} V_x$. 
The weight 
$$
\delta(x):=\frac{1}{2}\sum_{\stackrel{\alpha\in\Rgot_{x}^{+}\cap\theta(\Rgot_{x}^{+})}{\theta(\alpha)\neq\alpha}}\alpha \ 
$$
defines a character $\Cbb_{\delta(x)}$ of the abelian group $H_x$. Recall that $m_x\in\mathbb{N}$ corresponds to the quantity 
$ \frac{1}{2}| \Rgot_{x}^{+}\cap\theta(\Rgot_{x}^{+})\cap\{\theta(\alpha)\neq\alpha\}| + \dim \mathbb{E}^{\hbox{\scriptsize\rm nci}}_{x}$.

The following lemma will be proved in Section \ref{sec:Q-Hx-Lambda}.

\begin{lem}\label{lem:wedge-J-compare}
The following identity holds :
$$
\bigwedge_{J_{V_x}} V_x\simeq (-1)^{m_x}\, \Cbb_{\delta(x)}\otimes \det(\mathbb{E}_x^{\hbox{\scriptsize\rm nci}})\otimes\bigwedge_{-J_{\mu^+}} V_x.
$$
\end{lem}

On the vector space $V_x$, we can work with two localized Riemann-Roch characters:
\begin{itemize}
\item $\RR_{H_{x}}(V_{x},\Phi_x,\{0\})$ is defined with the complex structure $J_{V_x}$,
\item $\widetilde{\RR}_{H_{x}}(V_{x},\Phi_x,\{0\})$ is defined with the complex structure $-J_{\mu^+}$.
\end{itemize}
Thanks to the previous Lemma we know that $\RR_{H_{x}}(V_{x},\Phi_x,\{0\})$ is equal to 
$(-1)^{m_x}\, \Cbb_{\delta(x)}\otimes \det(\mathbb{E}_x^{\hbox{\scriptsize\rm nci}})\otimes \widetilde{\RR}_{H_{x}}(V_{x},\Phi_x,\{0\})$.

\begin{prop}
We have
\begin{equation}\label{eq:RR-V-x}
\RR_{H_{x}}(V_{x},\Phi_x,\{0\})=(-1)^{m_x} \Cbb_{\delta(x)} \otimes \det(\mathbb{E}_x^{\hbox{\scriptsize\rm nci}})\otimes {\rm Sym}(\qgot/\qgot_{\mu^+}).
\end{equation}
\end{prop}

\begin{proof}
For $s\in [0,1]$, we consider the $H_x$-equivariant map $\Phi^s:V_x\to\hgot_x^*$ defined by the relation  
$\Phi^s(v)= \mu^+ + s\Phi_{V_x}(v)$. The corresponding Kirwan vector field on $V_x$ is 
$\kappa^s(v)= -\Phi^s(v)\cdot v$.
It is not difficult to see that there exists $\epsilon>0$ such that $\{\kappa^s=0\}\cap\{\|v\|\leq \epsilon\}=\{0\}$ for any 
$s\in [0,1]$. Then a simple deformation argument gives that 
$\widetilde{\RR}_{H_{x}}(V_{x},\Phi^s,\{0\})$ does not depend on $s\in [0,1]$. We have proved that 
$$
\widetilde{\RR}_{H_{x}}(V_{x},\Phi_x,\{0\})=\widetilde{\RR}_{H_{x}}(V_{x},\mu^+,\{0\})
$$
where $\mu^+$ denotes the constant map $\Phi^0$. Standard computations  give 
$\widetilde{\RR}_{H_{x}}(V_{x},\mu^+,\{0\})$ $={\rm Sym}(\qgot/\qgot_{\mu^+})$ (see \cite{pep-RR}[Proposition 5.4]). Our proof is completed.  
\end{proof}

 \subsubsection{Conclusion}
If we use the formulas  (\ref{eq:holomorphic induction}), (\ref{eq:cotangent induction}) and (\ref{eq:RR-V-x}) we obtain
the following expression 
$$
Q_{Hx}(\lambda)=(-1)^{m_x} 
{\rm Ind}_{H_{x}}^H\left( \Cbb_{\lambda_x+\delta(x)} \otimes \det(\mathbb{E}_x^{\hbox{\scriptsize\rm  nci}})\otimes {\rm Sym}(\qgot/\qgot_{\mu^+})\otimes\bigwedge_\Cbb \hgot/\hgot_{\mu^+}\right)
$$
in $\wR(H)$. Here $\Cbb_{\lambda_x}$ is the character of $G_x$ associated to the weight $\lambda_x=g\lambda$.

The previous formula  depends on a choice of a regular element $r$ in the Weyl chamber. In the next section we will propose another expression for $Q_{Hx}(\lambda)$ that does not depend on this choice.

 \subsection{Another expression for $Q_{Hx}(\lambda)$}\label{sec:Q-Hx-Lambda}

Let $\Rgot_x\subset \ggot_x^*$ be the roots for the action of the torus $G_x$ on $\ggot\otimes \Cbb$. The involution $\theta:\tgot^*\to\tgot^*$ leaves 
the set $\Rgot_x$ invariant and a root $\alpha\in\Rgot_x$ is called 
{\em imaginary} if $\theta(\alpha)=\alpha$. We denote respectively by $\Rgot_x^{\hbox{\scriptsize ci}}$ and by $\Rgot_x^{\hbox{\scriptsize nci}}$ 
the subsets of {\em compact imaginary} and {\em non-compact imaginary} roots.

We choose a generic element $r\in\tgot^*_+$ such that $\mu^+=(g\cdot r)^+$ satisfies the following relation : for any $\alpha\in\Rgot_x$, we have 
$$
(\alpha,\mu^+)=0\Longleftrightarrow \theta(\alpha)=-\alpha.
$$

Notice that an imaginary roots $\alpha$ is positive if and only if $(\alpha,\mu^+)>0$.

\begin{defi}
We consider the subset  $\Agot_x\subset \Rgot_x$ defined by the following relations:
$$
\alpha\in \Agot_x \ \Longleftrightarrow\  \alpha(\mu^+)>0\quad {\rm and}\quad \theta(\alpha)\neq \alpha.
$$
\end{defi}

The involution $\theta$ defines a free action of $\Zbb_2$ on the set $\Agot_x$. We denote by $\Agot_x/\Zbb_2$ its quotient. 
For any $\alpha\in \Rgot_x$, we denote by $\Cbb_{\alpha}$ the corresponding $1$-dimensional representation of $G_x$, 
and $\Cbb_{\alpha}\vert_{H_x}$ its restriction to the subgroup $H_x$. We have a natural map 
$[\alpha]\in \Agot_x/\Zbb_2 \longmapsto \Cbb_{\alpha}\vert_{H_x}\in R(H_x)$.

For any $\alpha\in\Rgot_x$ we define 
$$
\tilde{\alpha}=\pm \alpha
$$
where $\pm$ is the sign of $\alpha(\mu)\alpha(\theta(\mu))$.

We consider the $H_x$-modules $\hgot/\hgot_{\mu^+}:=([\hgot,\mu^+], J_{\mu^+})$, $\qgot/\qgot_{\mu^+}:=(V_{x}, J_{\mu^+})$  and $(V_x, J_{V_x})$. 
\begin{lem}
We have the following isomorphisms of $H_x$-modules
\begin{eqnarray*}
\hgot/\hgot_{\mu^+}&\simeq& \bigoplus_{[\alpha]\in \Agot_x/\Zbb_2} \Cbb_{\alpha}\vert_{H_x}\oplus \bigoplus_{\alpha\in\Rgot_x^{\hbox{\scriptsize \rm ci}}\cap\Rgot^{+}_{x}}\Cbb_{\alpha}\vert_{H_x}\qquad \qquad[A],\\
\qgot/\qgot_{\mu^+}&\simeq& \bigoplus_{[\alpha]\in \Agot_x/\Zbb_2} \Cbb_{\alpha}\vert_{H_x}\oplus \bigoplus_{\alpha\in\Rgot_x^{\hbox{\scriptsize \rm nci}}\cap\Rgot^{+}_{x}} \Cbb_{\alpha}\vert_{H_x}\qquad \qquad[B],\\
(V_{x}, J_{V_x})&\simeq &\bigoplus_{[\alpha]\in \Agot_x/\Zbb_2} \Cbb_{\tilde{\alpha}}\vert_{H_x}\oplus \bigoplus_{\alpha\in\Rgot_x^{\hbox{\scriptsize \rm nci}}\cap\Rgot^{+}_{x}}
\Cbb_{\tilde{\alpha}}\vert_{H_x}\qquad \qquad [C].
\end{eqnarray*}
\end{lem}

\begin{proof}
Thanks to Lemma \ref{lem:V-x-complexe}, we know that the $H_x$-module $(V_{x}, J_{V_x})$ is isomorphic to the vector space $[\qgot,{\mu^+}]$ equipped with the complex structure $J_{V_x}:= J_{\mu^+}\circ S_x$.
We consider the vector spaces $[\qgot, \mu^+]$ and $[\ggot, \mu^+]$ equipped with the complex structure $J_{\mu^+}$. 
The projection (taking the real part)
${\bf r}:\ggot\otimes \Cbb\to \ggot$ induces an isomorphism of $G_x$-modules
$$
{\bf r} : \bigoplus_{\alpha(\mu^+)>0} (\ggot\otimes \Cbb)_\alpha\quad \longrightarrow\quad [\ggot, \mu^+].
$$

The orthogonal projections ${\bf p}_1: [\ggot, \mu^+]\to  [\qgot, \mu^+]$ and ${\bf p}_2: [\ggot, \mu^+]\to  [\hgot, \mu^+]$ commute with the $H_x$-action, so 
the maps 
\begin{eqnarray*}
{\bf p}_1\circ {\bf r}:\bigoplus_{\alpha(\mu^+)>0} (\ggot\otimes \Cbb)_\alpha &\longrightarrow&   [\qgot, \mu^+],\\
{\bf p}_2\circ {\bf r}:\bigoplus_{\alpha(\mu^+)>0} (\ggot\otimes \Cbb)_\alpha &\longrightarrow&  [\hgot, \mu^+]
\end{eqnarray*}
are surjective morphisms of $H_x$-modules.

Let $V^1_x(\alpha)={\bf p}_1\circ {\bf r}((\ggot\otimes \Cbb)_\alpha)$. We notice that $\dim_\Cbb V^1_x(\alpha)\in \{0,1\}$:  
$V^1_x(\alpha)=\{0\}$ only if $\alpha$ is a non-compact imaginary root and $V^1_x(\alpha)\simeq \Cbb_{\alpha}\vert_{H_x}$ when $V^1_x(\alpha)\neq \{0\}$. 
We notice also that $V^1_x(\alpha)=V^1_x(\theta(\alpha))$, hence 
$$
\qgot/\qgot_{\mu^+}=([\qgot,{\mu^+}],J_{\mu^+})\simeq\bigoplus_{[\alpha]\in \Agot_x/\Zbb_2} V^1_x(\alpha)\oplus \bigoplus_{\alpha\in\Rgot_x^{\hbox{\scriptsize nci}}\cap\Rgot^{+}_{x}} V^1_x(\alpha).
$$
The identity $[B]$ is proved.

Similarly we consider $V^2_x(\alpha)={\bf p}_2\circ {\bf r}((\ggot\otimes \Cbb)_\alpha)$. We notice that $\dim_\Cbb V^2_x(\alpha)\in \{0,1\}$: 
$V^2_x(\alpha)=\{0\}$ only if $\alpha$ is a compact imaginary root and $V^2_x(\alpha)\simeq \Cbb_{\alpha}\vert_{H_x}$ when $V^2_x(\alpha)\neq \{0\}$. 
We notice also that $V^2_x(\alpha)=V^2_x(\theta(\alpha))$, hence 
$$
\hgot/\hgot_{\mu^+}=([\hgot,{\mu^+}],J_{\mu^+})\simeq\bigoplus_{[\alpha]\in \Agot_x/\Zbb_2} V^2_x(\alpha)\oplus \bigoplus_{\alpha\in\Rgot_x^{\hbox{\scriptsize ci}}\cap\Rgot^{+}_{x}} V^2_x(\alpha).
$$
The identity $[A]$ is proved.

Finally we check that the complex structures $J_{\mu^+}$ and $J_{V_x}$ preserve each $V^1_x(\alpha)$ and that $(V_x(\alpha),J_{V_x})\simeq \Cbb_{\tilde\alpha}\vert_{H_x}$ when 
$(\alpha,\mu^+)>0$. The identity $[C]$ follows.  
\end{proof}

\medskip

We consider the $H_x$-module $\mathbb{V}_x:= \sum_{[\alpha]\in \Agot_x/\Zbb_2}\Cbb_{\alpha}\vert_{H_x}$, 
and the $G_x$-modules 
$\mathbb{E}^{\hbox{\scriptsize nci}}_{x}:=\sum_{\alpha\in\Rgot_x^{\hbox{\scriptsize nci}}\cap\Rgot^{+}_{x}}\Cbb_{\alpha}$ and 
$\mathbb{E}^{\hbox{\scriptsize ci}}_{x}:=\sum_{\alpha\in\Rgot_x^{\hbox{\scriptsize ci}}\cap\Rgot^{+}_{x}}\Cbb_{\alpha}$. 
In the previous lemma we have proved that $H_x$-modules $\hgot/\hgot_{\mu^+}$ and $\qgot/\qgot_{\mu^+}$ are respectively isomorphic to 
$\mathbb{V}_x\oplus \mathbb{E}^{\hbox{\scriptsize ci}}_{x}$ and  $\mathbb{V}_x\oplus \mathbb{E}^{\hbox{\scriptsize nci}}_{x}$. If we use the fact that 
${\rm Sym}(\mathbb{V}_x)\otimes \bigwedge \mathbb{V}_x=1$, we get the 
following corollary.

\begin{coro}\label{coro:simplification}
We have the following identity of virtual $H_x$-modules:
$$
{\rm Sym}(\qgot/\qgot_{\mu^+})\otimes \bigwedge \hgot/\hgot_{\mu^+}\simeq {\rm Sym}(\mathbb{E}^{\hbox{\scriptsize \rm nci}}_{x})\otimes \bigwedge \mathbb{E}^{\hbox{\scriptsize \rm ci}}_{x}.
$$
\end{coro}

\medskip

{\em Proof of Lemma \ref{lem:wedge-J-compare}.}  Let $\Bcal:=\Agot_x/\Zbb_2\bigcup (\Rgot_x^{\hbox{\scriptsize nci}}\cap\Rgot^{+}_{x})$. We see that 
$\bigwedge_{J_{V_x}} V_x =$\break $ \prod_{\alpha\in \Bcal}(1-t^{\tilde{\alpha}})$ whereas $\bigwedge_{-J_{\mu^+}} V_x = \prod_{\alpha\in \Bcal}(1-t^{-\alpha})$. Accordingly 
we get $\bigwedge_{J_{V_x}} V_x\simeq (-1)^{|\Bcal' |}\, \Cbb_{\eta}\otimes \bigwedge_{-J_{\mu^+}} V_x$ where $\Bcal'=\{\alpha\in \Bcal,\tilde{\alpha}=\alpha\}$ and 
$\eta=\sum_{\alpha\in\Bcal'}\alpha$. Now it is easy to check that an element $\alpha\in\Bcal$ belongs to $\Bcal'$ if and only if $\alpha$ and $\theta(\alpha)$ both belong to 
$\Rgot_x^+$. In other words 
$$
\Bcal'= \left\{\alpha\in\Rgot_{x}^{+}\cap\theta(\Rgot_{x}^{+}),\ \theta(\alpha)\neq\alpha\right\}\slash \Zbb_2 \bigcup\Rgot_x^{\hbox{\scriptsize nci}}\cap\Rgot^{+}_{x}. 
$$
We have proved that 
$$
\bigwedge_{J_{V_x}} V_x\simeq (-1)^{m_x}\, \Cbb_{\delta(x)}\otimes\det(\mathbb{E}^{\hbox{\scriptsize \rm nci}}_{x})\otimes \bigwedge_{-J_{\mu^+}} V_x.
$$
$\Box$

\medskip

Finally, thanks to Lemma \ref{lem:wedge-J-compare} and Corollary \ref{coro:simplification}, we obtain the final formula for $Q_{Hx}(\lambda)$ (that does not depend on the choice of $r$): 
$$
Q_{Hx}(\lambda)=(-1)^{m_x} 
{\rm Ind}_{H_{x}}^H\left(\Cbb_{\lambda_x+\delta(x)} \otimes\det(\mathbb{E}^{\hbox{\scriptsize \rm nci}}_{x})\otimes  {\rm Sym}(\mathbb{E}^{\hbox{\scriptsize \rm nci}}_{x})\otimes\bigwedge \mathbb{E}^{\hbox{\scriptsize ci}}_{x}\right).
$$

\subsection{Computation of the virtual module  $\mathbb{M}_x(\lambda)$} \label{sec:formule-M-x}

According to Theorem \ref{theo-principal}, we have the decomposition $V_{\lambda}^{G}\vert_{H}=\sum_{\bar{x}} Q_{\bar{x}}(\lambda)$
where $Q_{\bar{x}}(\lambda)={\rm Ind}^{H}_{H_{x}}\left( \mathbb{A}_{x}(\lambda)\right)$, and $\mathbb{A}_{x}(\lambda)\in \wR(H_x)$ has the following description
$$
\mathbb{A}_{x}(\lambda)=
\frac{1}{|W^H_x|}\sum_{w\in W}(-1)^{m_{xw}} \Cbb_{\lambda_{xw}+\delta(xw)}
\otimes \det(\mathbb{E}^{\hbox{\scriptsize\rm nci}}_{xw})\otimes {\rm Sym}(\mathbb{E}^{\hbox{\scriptsize\rm nci}}_{xw}) \otimes \bigwedge 
\mathbb{E}^{\hbox{\scriptsize\rm  ci}}_{xw}.
$$

The aim of this section is to simplify the expression of the virtual $H_x$-module $\mathbb{A}_{x}(\lambda)$. We start by comparing the $G_x$-modules $\mathbb{E}^{\hbox{\scriptsize\rm nci}}_{xw}$ and 
$\mathbb{E}^{\hbox{\scriptsize\rm nci}}_{x}$. We use the decomposition 
$\mathbb{E}^{\hbox{\scriptsize\rm nci}}_{x}=\left(\mathbb{E}^{\hbox{\scriptsize\rm nci}}_{x}\right)^+_w\oplus \left(\mathbb{E}^{\hbox{\scriptsize\rm nci}}_{x}\right)^-_w$
where 
$$
\left(\mathbb{E}^{\hbox{\scriptsize\rm nci}}_{x}\right)^+_w:=\sum_{\alpha\in \Rgot_x^{\hbox{\tiny \rm nci}}\cap\Rgot^{+}_x\cap \Rgot^{+}_{xw}}\Cbb_{\alpha},\quad{\rm and}\quad
\left(\mathbb{E}^{\hbox{\scriptsize\rm nci}}_{x}\right)^-_w=\sum_{\alpha\in \Rgot_x^{\hbox{\tiny \rm nci}}\cap\Rgot^{+}_x\cap -\Rgot^{+}_{xw}}\Cbb_{\alpha}.
$$

We have the following basic lemma (see Lemma \ref{lem:formulation-M-lambda}).

\begin{lem}\label{lem:polarized-module} The $G_x$-module $\vert\mathbb{E}^{\hbox{\scriptsize\rm nci}}_{x}\vert_w:=\left(\mathbb{E}^{\hbox{\scriptsize\rm nci}}_{x}\right)^+_w\oplus \overline{\left(\mathbb{E}^{\hbox{\scriptsize\rm nci}}_{x}\right)^-_w}$ is isomorphic to 
$\mathbb{E}^{\hbox{\scriptsize\rm nci}}_{xw}$.

\end{lem}

Let $\rho=\tfrac{1}{2}\sum_{\alpha\in\Rgot^+}\alpha$. We denote by $w\bullet\lambda=w(\lambda+\rho)-\rho$ the affine action of the Weyl group on the lattice 
$\Lambda$.

The main result of this section is the following proposition.

\begin{prop}\label{prop:formule-M-lambda} Let $x\in Z_\theta$. We have 
$$
\mathbb{A}_{x}(\lambda)=\mathbb{M}_x(\lambda)\otimes\Cbb_{\delta(x)} \otimes \bigwedge 
\mathbb{E}^{\hbox{\scriptsize\rm  ci}}_{x}
$$
 where $\mathbb{M}_x(\lambda)\in \wR(H_x)$ is defined by the following expression

$$
\mathbb{M}_{x}(\lambda)=
\frac{(-1)^{n_x}}{|W^H_x|}\sum_{w\in W}(-1)^{k_{x,w}}\ \Cbb_{(w\bullet\lambda)_{x}}\otimes \det(\left(\mathbb{E}^{\hbox{\scriptsize\rm nci}}_{x}\right)^+_w)
\otimes {\rm Sym}(|\mathbb{E}^{\hbox{\scriptsize\rm nci}}_{x}|_w),
$$
and
\begin{itemize}
\item $k_{x,w}= |\Rgot^{+}_x\cap \Rgot^{+}_{xw}\cap\{\theta(\alpha)\neq \pm \alpha\}| +|\Rgot^{+}_x\cap \Rgot^{+}_{xw}\cap \Rgot_x^{\hbox{\scriptsize \rm ci}} |$,
\item $n_x:=| \theta(\Rgot^{+}_x)\cap \Rgot^{+}_x| - \frac{1}{2}| \theta(\Rgot^{+}_x)\cap \Rgot^{+}_x\cap\{\theta(\alpha)\neq\alpha\}|$.
\end{itemize}
\end{prop}

\begin{rem}
We can describe $Q_{\bar{x}}(\lambda)$ differently by taking $\{w_1,\cdots,w_p\}\subset W$ such that $W^H_x\backslash W\simeq \{\bar{w}_1,\cdots,\bar{w}_p\}$. We have 
$Q_{\bar{x}}(\lambda)={\rm Ind}^{H}_{H_{x}}\left( \tilde{\mathbb{A}}_{x}(\lambda)\right)$ with
$$
\tilde{\mathbb{A}}_{x}(\lambda)=\tilde{\mathbb{M}}_x(\lambda)\otimes\Cbb_{\delta(x)}\otimes\bigwedge 
\mathbb{E}^{\hbox{\scriptsize\rm  ci}}_{x}
$$
and where $\tilde{\mathbb{M}}_x(\lambda)\in \wR(H_x)$ is defined by the following expression
$$
\tilde{\mathbb{M}}_{x}(\lambda)=(-1)^{n_{x}}\,\sum_{k=1}^p
(-1)^{k_{x,w_k}}\ \Cbb_{(w_k\bullet\lambda)_{x}}\otimes \det(\left(\mathbb{E}^{\hbox{\scriptsize\rm nci}}_{x}\right)^+_{w_k})
\otimes {\rm Sym}(|\mathbb{E}^{\hbox{\scriptsize\rm nci}}_{x}|_{w_k}).
$$
\end{rem}

\medskip

We need to introduce some notations. To $x\in Z_\theta$, we associate~:

\begin{itemize}

\item The polarized roots : to $\alpha\in\Rgot_x$ and $w\in W$, we associate $|\alpha|_w\in\Rgot_x$ defined as follows
$$
|\alpha|_w=
\begin{cases}
\alpha \qquad \ \ {\rm if} \ \ \alpha \in \Rgot_{xw}^+,\\
-\alpha \qquad {\rm if} \ \ \alpha \notin \Rgot_{xw}^+.
\end{cases}
$$

\item The following $G_x$-weights :
$$
\gamma^{\hbox{\scriptsize \rm ci}}_{x,w}:=\sum_{\stackrel{\alpha\in \Rgot_x^{\hbox{\tiny \rm ci}}\cap \Rgot_x^+}{|\alpha|_w\neq\alpha}}\alpha, \quad
\gamma^{\hbox{\scriptsize \rm nci}}_{x,w}:=\sum_{\stackrel{\alpha\in \Rgot_x^{\hbox{\tiny \rm nci}}\cap \Rgot_x^+}{|\alpha|_w\neq\alpha}}\alpha, \quad 
\gamma_{x,w}:=\sum_{\stackrel{\alpha\in \Rgot_x^+}{|\alpha|_w\neq\alpha}}\alpha.
$$

\end{itemize}
\medskip

The proof of Proposition \ref{prop:formule-M-lambda} is based on the following Lemma. 

\begin{lem}\label{lem:formulation-M-lambda} Let $x\in Z_\theta$ and $w\in W$. Let $d^{\hbox{\scriptsize \rm ci}}_{x,w}$ be the cardinal of the set \break
$\{\alpha\in \Rgot_x^{\hbox{\scriptsize \rm ci}}\cap \Rgot_x^+ ,\ |\alpha|_w\neq\alpha\}$. We have the following relations
\begin{enumerate}
\item 
$
\mathbb{E}^{\hbox{\scriptsize \rm nci}}_{xw}\simeq \sum_{\alpha\in \Rgot_x^{\hbox{\tiny \rm nci}}\cap \Rgot_x^+}\Cbb_{|\alpha|_w}
\quad {\rm and} \quad 
\mathbb{E}^{\hbox{\scriptsize \rm ci}}_{xw}\simeq \sum_{\alpha\in \Rgot_x^{\hbox{\tiny \rm ci}}\cap \Rgot_x^+}\Cbb_{|\alpha|_w},
$

\item $\det(\mathbb{E}^{\hbox{\scriptsize \rm nci}}_{xw})=\Cbb_{-\gamma^{\hbox{\tiny \rm nci}}_{x,w}}\otimes \det(\left(\mathbb{E}^{\hbox{\scriptsize\rm nci}}_{x}\right)^+_w)$,

\item $\bigwedge \mathbb{E}^{\hbox{\scriptsize \rm ci}}_{xw}=
(-1)^{d^{\hbox{\tiny \rm ci}}_{x,w} }\,\Cbb_{-\gamma^{\hbox{\tiny \rm ci}}_{x,w}}\otimes 
\bigwedge \mathbb{E}^{\hbox{\scriptsize \rm ci}}_x$.
\item The $H_x$-weight $\delta(xw)-\delta(x)$ is equal to the restriction of  $G_x$-weight
$\gamma^{\hbox{\scriptsize \rm nci}}_{x,w}+\gamma^{\hbox{\scriptsize \rm ci}}_{x,w}- \gamma_{x,w}$ to $H_x$.

\end{enumerate}

\end{lem}

\begin{proof} 
We remark that $\Rgot_x=\Rgot_{xw}$, $\Rgot_x^+=g(\Rgot^+)$ and $\Rgot_{xw}^+=g(w\Rgot^+)$. The first point follows and points (ii) and (iii) derive from the first. 

Let us check the last point. The term $\rho_x:=\frac{1}{2}\sum_{\alpha\in\Rgot_x^+}\alpha$ is the image of 
$\rho:=\frac{1}{2}\sum_{\alpha'\in\Rgot^+}\alpha'$ through the map $\mu\mapsto\mu_x$. We see that 
$$
\rho_x+\theta(\rho_x)=\sum_{\alpha\in\Rgot_x^+\cap\theta(\Rgot_x^+)}\alpha=2\delta(x)+2\rho^{\hbox{\scriptsize \rm nci}}_x+2\rho^{\hbox{\scriptsize \rm ci}}_x
$$
where $\rho^{\hbox{\scriptsize \rm nci}}_x=\frac{1}{2}\sum_{\alpha\in\Rgot^{\hbox{\tiny \rm nci}}_x\cap\Rgot_x^+}\alpha$ and $\rho^{\hbox{\scriptsize \rm ci}}_x=
\frac{1}{2}\sum_{\alpha\in\Rgot^{\hbox{\tiny \rm ci}}_x\cap\Rgot_x^+}\alpha$. Similarly we have 
$$
\rho_{xw}+\theta(\rho_{xw})=2\delta(xw)+2\rho^{\hbox{\scriptsize \rm nci}}_{xw}+2\rho^{\hbox{\scriptsize \rm ci}}_{xw}.
$$

Thus the $H_x$-weight $\delta(xw)-\delta(x)$ is equal to the restriction to $H_x$ of the $G_x$-weight
$$
\beta(x,w):= \rho_{xw}-\rho_{x}+ (\rho^{\hbox{\scriptsize \rm nci}}_{x}-\rho^{\hbox{\scriptsize \rm nci}}_{xw})+
(\rho^{\hbox{\scriptsize \rm ci}}_{x}-\rho^{\hbox{\scriptsize \rm ci}}_{xw}).
$$
We notice that $ \rho_{xw}-\rho_{x}= (w\rho-\rho)_x= -\gamma_{x,w}$. Furthermore, small computations give that $\rho^{\hbox{\scriptsize \rm nci}}_{x}-\rho^{\hbox{\scriptsize \rm nci}}_{xw}= \gamma^{\hbox{\scriptsize \rm nci}}_{x,w}$ and 
$\rho^{\hbox{\scriptsize \rm ci}}_{x}-\rho^{\hbox{\scriptsize \rm ci}}_{xw}= \gamma^{\hbox{\scriptsize \rm ci}}_{x,w}$. We have proved that 
$\beta(x,w)= \gamma^{\hbox{\scriptsize \rm nci}}_{x,w}+\gamma^{\hbox{\scriptsize \rm ci}}_{x,w}- \gamma_{x,w}$. The last point follows.  
\end{proof}

\bigskip

Now, we can finish the proof of the Proposition \ref{prop:formule-M-lambda}.  We must check that the virtual $H_x$-module
$$
{\bf A}:=(-1)^{m_{xw}} \Cbb_{\lambda_{xw}+\delta(xw)}\otimes\det(\mathbb{E}^{\hbox{\scriptsize\rm nci}}_{xw})
\otimes \bigwedge \mathbb{E}^{\hbox{\scriptsize\rm  ci}}_{xw}
$$
is equal to the virtual $H_x$-module
$$
{\bf B}:=(-1)^{n_x+k_{x,w}}\ \Cbb_{(w\bullet\lambda)_{x}+\delta(x)}\otimes \det(\left(\mathbb{E}^{\hbox{\scriptsize\rm nci}}_{x}\right)^+_w)\otimes \bigwedge \mathbb{E}^{\hbox{\scriptsize\rm  ci}}_{x}.
$$

If we use Lemma \ref{lem:formulation-M-lambda}, we get
$$
{\bf A}=
(-1)^{m_{xw}+d^{\hbox{\tiny \rm ci}}_{x,w}} \Cbb_{(w(\lambda +\rho)-\rho)_x +\delta(x)}\otimes \det(\left(\mathbb{E}^{\hbox{\scriptsize\rm nci}}_{x}\right)^+_w)\otimes \bigwedge \mathbb{E}^{\hbox{\scriptsize\rm  ci}}_{x}.
$$
Thus the equality ${\bf A}={\bf B}$ follows from the following lemma.

\begin{lem}
For any $x\in Z_\theta$ and $w\in W$, we have $n_x+k_{x,w}=m_{xw}+d^{\hbox{\tiny \rm ci}}_{x,w}\ {\rm mod}\ 2$.
\end{lem}

\begin{proof} In order to simplify our notations, we write $a\equiv b$ for $a=b \ {\rm mod}\ 2$. 

We have $\dim \mathbb{E}^{\hbox{\scriptsize\rm nci}}_{x}= \dim \mathbb{E}^{\hbox{\scriptsize\rm nci}}_{xw}$ and 
$\dim \mathbb{E}^{\hbox{\scriptsize\rm ci}}_{x}= \dim \mathbb{E}^{\hbox{\scriptsize\rm ci}}_{xw}$, then
\begin{eqnarray*}
m_{xw}-m_x&=& \frac{1}{2}\left(| \Rgot_{xw}^{+}\cap\theta(\Rgot_{xw}^{+})\cap\{\theta(\alpha)\neq\alpha\}| - 
| \Rgot_{x}^{+}\cap\theta(\Rgot_{x}^{+})\cap\{\theta(\alpha)\neq\alpha\}|\right)\\
&=& \frac{1}{2}\left(| \Rgot_{xw}^{+}\cap\theta(\Rgot_{xw}^{+})| - 
| \Rgot_{x}^{+}\cap\theta(\Rgot_{x}^{+})|\right).
\end{eqnarray*}
We remark now that 
$$
\Rgot_{xw}^{+}\cap\theta(\Rgot_{xw}^{+})= A_{++}\cup A_{--}\cup A_{+-}\cup A_{-+}
$$
with $A_{++}=\Rgot_{x}^{+}\cap\theta(\Rgot_{x}^{+})\cap \Rgot_{xw}^{+}\cap\theta(\Rgot_{xw}^{+})$, 
$A_{--}=-\Rgot_{x}^{+}\cap-\theta(\Rgot_{x}^{+})\cap \Rgot_{xw}^{+}\cap\theta(\Rgot_{xw}^{+})$,
$A_{+-}=\Rgot_{x}^{+}\cap\theta(-\Rgot_{x}^{+})\cap \Rgot_{xw}^{+}\cap\theta(\Rgot_{xw}^{+})$ and 
$A_{-+}=-\Rgot_{x}^{+}\cap\theta(\Rgot_{x}^{+})\cap \Rgot_{xw}^{+}\cap\theta(\Rgot_{xw}^{+})$.

Similarly we have 
$$
\Rgot_{x}^{+}\cap\theta(\Rgot_{x}^{+})= B_{++}\cup B_{--}\cup B_{+-}\cup B_{-+}
$$
with $B_{++}=\Rgot_{x}^{+}\cap\theta(\Rgot_{x}^{+})\cap \Rgot_{xw}^{+}\cap\theta(\Rgot_{xw}^{+})$, 
$B_{--}=\Rgot_{x}^{+}\cap\theta(\Rgot_{x}^{+})\cap -\Rgot_{xw}^{+}\cap\theta(-\Rgot_{xw}^{+})$,
$B_{+-}=\Rgot_{x}^{+}\cap\theta(\Rgot_{x}^{+})\cap \Rgot_{xw}^{+}\cap\theta(-\Rgot_{xw}^{+})$ and 
$B_{-+}=\Rgot_{x}^{+}\cap\theta(\Rgot_{x}^{+})\cap -\Rgot_{xw}^{+}\cap\theta(\Rgot_{xw}^{+})$.

We have the obvious relations : $A_{++}=B_{++}$, $A_{--}=-B_{--}$, $\theta(A_{+-})=A_{-+}$, $\theta(B_{+-})=B_{-+}$ and 
$A_{++}=B_{++}$. So we get $m_{xw}-m_x\equiv |A_{+-}| +|B_{+-}|$.

Let consider $\Acal:=\Rgot_{x}^{+}\cap \Rgot_{xw}^{+}$ and $\Bcal:= \Rgot_{x}^{+}\cap -\Rgot_{xw}^{+}$. We have
\begin{eqnarray*}
m_{xw}-m_x &\equiv& |\Acal\cap \theta(\Bcal)|+ |\Acal\cap -\theta(\Bcal)| \\
&\equiv& |\Acal | +  |\Acal\cap \theta(\Acal)|+ |\Acal\cap -\theta(\Acal)|.
\end{eqnarray*}
Now we remark that 
\begin{eqnarray*}
|\Acal\cap \theta(\Acal)| &\equiv& |\Acal\cap \theta(\Acal)\cap\{\theta(\alpha)= \alpha\}| \\
&\equiv& |\Rgot_{x}^{+}\cap \Rgot_{xw}^{+}\cap\{\theta(\alpha)= \alpha\}|.
\end{eqnarray*}
Similarly 
\begin{eqnarray*}
|\Acal\cap -\theta(\Acal)| &\equiv& |\Acal\cap -\theta(\Acal)\cap\{\theta(\alpha)= -\alpha\}|\\
&\equiv& |\Rgot_{x}^{+}\cap \Rgot_{xw}^{+}\cap\{\theta(\alpha)= -\alpha\}| .  
\end{eqnarray*}

At this stage we have proved that
\begin{eqnarray*}
m_{xw}-m_x& \equiv & | \Rgot_{x}^{+}\cap \Rgot_{xw}^{+}| + 
|\Rgot_{x}^{+}\cap \Rgot_{xw}^{+}\cap\{\theta(\alpha)= \alpha\}|  +
|\Rgot_{x}^{+}\cap \Rgot_{xw}^{+}\cap\{\theta(\alpha)= -\alpha\}|\\
&\equiv& |\Rgot_{x}^{+}\cap \Rgot_{xw}^{+}\cap\{\theta(\alpha)\neq -\alpha\}| + |\Rgot_{x}^{+}\cap \Rgot_{xw}^{+}\cap\{\theta(\alpha)= \alpha\}|.
\end{eqnarray*}

As $d^{\hbox{\scriptsize \rm ci}}_{x,w}= |\Rgot_{x}^{+}\cap -\Rgot_{xw}^{+}\cap \Rgot_x^{\hbox{\scriptsize ci}}|$, we have 
$|\Rgot_{x}^{+}\cap \Rgot_{xw}^{+}\cap\{\theta(\alpha)= \alpha\}|+ d^{\hbox{\scriptsize \rm ci}}_{x,w}$ is equal to 
$\dim \mathbb{E}^{\hbox{\scriptsize ci}}_x+ |\Rgot_{x}^{+}\cap \Rgot_{xw}^{+}\cap \Rgot_x^{\hbox{\scriptsize nci}}|$. This implies that 
$m_{xw}+ d^{\hbox{\scriptsize \rm ci}}_{x,w}$ is equal, modulo $2$, to 
$$
m_x + \dim \mathbb{E}^{\hbox{\scriptsize ci}}_x+ |\Rgot_{x}^{+}\cap \Rgot_{xw}^{+}\cap \Rgot_x^{\hbox{\scriptsize nci}}|
+ |\Rgot_{x}^{+}\cap \Rgot_{xw}^{+}\cap\{\theta(\alpha)\neq -\alpha\}|
$$
$$
\equiv m_x + \dim \mathbb{E}^{\hbox{\scriptsize ci}}_x+ |\Rgot_{x}^{+}\cap \Rgot_{xw}^{+}\cap \Rgot_x^{\hbox{\scriptsize ci}}|
+ |\Rgot_{x}^{+}\cap \Rgot_{xw}^{+}\cap\{\theta(\alpha)\neq \pm \alpha\}|.
$$
By definition $m_x=\frac{1}{2}| \Rgot_{x}^{+}\cap\theta(\Rgot_{x}^{+})\cap\{\theta(\alpha)\neq\alpha\}| + \dim \mathbb{E}^{\hbox{\scriptsize\rm nci}}_{x}$ and then
\begin{eqnarray*}
m_x + \dim \mathbb{E}^{\hbox{\scriptsize ci}}_x
&\equiv& \frac{1}{2}| \Rgot_{x}^{+}\cap\theta(\Rgot_{x}^{+})\cap\{\theta(\alpha)\neq\alpha\}|+ 
|\Rgot_{x}^{+}\cap \{\theta(\alpha)=\alpha\}|\\
&\equiv& n_x.
\end{eqnarray*}

Finally we have proved that $m_{xw}+ d^{\hbox{\scriptsize \rm ci}}_{x,w}$ is equal, modulo $2$, to $n_x+k_{x,w}$.

\end{proof}

\section{Examples}

In this section we will study in details some examples of our formula 
$$
V_{\lambda}^{G}\vert_{H}=\sum_{\bar{x}\,\in\, H\backslash Z_{\theta}\slash W} Q_{\bar{x}}(\lambda)
$$
where $Q_{\bar{x}}(\lambda)={\rm Ind}^{H}_{H_{x}}\left(\mathbb{M}_x(\lambda)\otimes\Cbb_{\delta(x)} \otimes \bigwedge 
\mathbb{E}^{\hbox{\scriptsize\rm  ci}}_{x}\right)$ and
$$
\mathbb{M}_{x}(\lambda)=
\frac{(-1)^{n_x}}{|W^H_x|}\sum_{w\in W}(-1)^{k_{x,w}}\ \Cbb_{(w\bullet\lambda)_{x}}\otimes \det(\left(\mathbb{E}^{\hbox{\scriptsize\rm nci}}_{x}\right)^+_w)
\otimes {\rm Sym}(|\mathbb{E}^{\hbox{\scriptsize\rm nci}}_{x}|_w).
$$

Here the integers $k_{x,w}$ and $n_x$ are defined as follows:
\begin{itemize}
\item $k_{x,w}= |\Rgot^{+}_x\cap \Rgot^{+}_{xw}\cap\{\theta(\alpha)\neq \pm \alpha\}| +|\Rgot^{+}_x\cap \Rgot^{+}_{xw}\cap \Rgot_x^{\hbox{\scriptsize \rm ci}} |$,
\item $n_x=| \theta(\Rgot^{+}_x)\cap \Rgot^{+}_x| - \frac{1}{2}| \theta(\Rgot^{+}_x)\cap \Rgot^{+}_x\cap\{\theta(\alpha)\neq\alpha\}|$.
\end{itemize}

\subsection{$K\subset K\times K$}

Let $K$ be a connected compact Lie group. Here we work with the Lie group $G=K\times K$ and the involution $\theta(k_1,k_2)=(k_2,k_1)$. 
The subgroup $H=G^\theta$ is the group $K$ embedded diagonally in $G$. 

Let $T$ be a maximal torus of $K$ and let $W_K=N_K(T)/T$ be the Weyl group of $K$. We denote by $\Rgot_K$ the set of roots for $(K,T)$, and we make the choice of a set $\Rgot_K^+$ of positive roots.

In the next lemma we describe the critical set $Z_\theta$ in the flag manifold $\Fcal=K/T\times K/T$ of $G$.

\begin{lem}
We have $Z_\theta= \bigcup_{w\in W_K} Z_w$ with $Z_w=K\cdot(wT,T)$. In other words, the set $H\backslash Z_{\theta}\slash W$ is a singleton.
\end{lem}

\begin{proof}
The element $x=(aT,bT)\in \Fcal$ belongs to $Z_\theta$ if and only if $=(a^{-1}b,b^{-1}a)\in W\times W$. If $b^{-1}a=w\in W$ then 
$(aT,bT)\in Z_w$.  
\end{proof}

\medskip

We take $x=(T,T)\in Z_\theta$. For each $w\in W_K$, we write $xw=(wT,T)$. We take 
$\lambda=(a,b)\in \Lambda_K^+\times\Lambda^+_K=\widehat{G}$.

Our data are as follows:

$\bullet$ the group $G_{x}$ is the maximal torus $T\times T\subset K$,

$\bullet$ the group $H_{x}$ is the maximal torus $T\subset K$,

$\bullet$ $\Cbb_{(w\bullet\lambda)_x+\delta(x)}=\Cbb_{w(a+\rho)+b+\rho}$ as a character of $T$,

$\bullet$ $n_{x}=|\Rgot_K^+|$,

$\bullet$ $k_{x,w}$ is equal to $|w\Rgot_K^+\cap\Rgot_K^+|+ |\Rgot_K^+|$, so $(-1)^{k_{x,w}}=(-1)^w$,

$\bullet$ the vector spaces $\mathbb{E}^{\hbox{\scriptsize\rm ci}}_{x},\mathbb{E}^{\hbox{\scriptsize\rm nci}}_{x}$ are reduced to $\{0\}$.

\medskip

In this context we obtain the following relation
\begin{equation}\label{eq:general-Clebsch-Gordan}
V^K_a\otimes V_{b}^K=(-1)^{\dim(K/T)/2}\sum_{w\in W_K}(-1)^w\, \ind_T^K\left(\Cbb_{w(a+\rho)+b+\rho} \right).
\end{equation}
This type of generalized Clebsch-Gordan formula what first noticed by Steinberg 
\cite{Steinberg-61} (see Section \ref{sec:Kostant}).

\medskip

\begin{example}\label{SU(2)-Clebsch-Gordan}
The irreducible representation $SU(2)$ are parametrized by $\Nbb$. If $n\geq 0$, the irreducible representation 
$V_n$ of $SU(2)$ satisfies 
$$
V_n=\ind_{U(1)}^{SU(2)}((\Cbb_0-\Cbb_2)\otimes\Cbb_n).
$$
If we take $m\geq n\geq 0$, then (\ref{eq:general-Clebsch-Gordan}) gives
\begin{eqnarray*}
V_n\otimes V_m&=& \ind_{U(1)}^{SU(2)}(\Cbb_{m-n}) - \ind_{U(1)}^{SU(2)}(\Cbb_{m+n+2})\\
&=& \sum_{k=0}^n\ind_{U(1)}^{SU(2)}((\Cbb_0-\Cbb_2)\otimes\Cbb_{m+n-2k})\\
&=& \sum_{k=0}^n \, V_{m+n-2k}.
\end{eqnarray*}
We recognize here the classical Clebsch-Gordan relations.
\end{example}

\subsection{$U(p)\times U(q)\subset U(p+q)$}

Let $p\geq q\geq 1$ and $n=p+q$. We take $G=U(n)$ with maximal torus $T\simeq U(1)^{n}$ the subgroup formed by the diagonal matrices. 
We use the canonical map $\tau$ from the symmetric group $\mathfrak{S}_n$ into $G$. It induces an isomorphism between $\mathfrak{S}_n$ and the 
Weyl group $W$ of $G$.

We work with the involution $\theta(g)=\Delta g \Delta^{-1}$ where $\Delta:={\rm diag}(I_p,-I_q)$: the subgroup fixed by $\theta$ is $H=U(p)\times U(q)$. 

In the next section we describe the critical set $Z_\theta \subset \Fcal$. For another type of parametrization of 
$H_\Cbb\backslash \Fcal$, see Section 5 of \cite{Richardson-Springer93}.

\subsubsection{The  critical set}

We consider the following elements of $O(2)$:
$$ 
R= \left(\begin{array}{cc} 
\tfrac{1}{\sqrt{2}}& \tfrac{1}{\sqrt{2}}\\ 
\tfrac{-1}{\sqrt{2}} & \frac{1}{\sqrt{2}}
\end{array}
\right),
\quad
S=\left(\begin{array}{cc} 0& 1\\ 1& 0\end{array}\right),
\quad 
J= \left(\begin{array}{cc} 
0& -1\\ 
1 &0
\end{array}
\right).
$$

The element $R$ is of order $8$, $R^2=-J$ and 
$R^{-1}\left(\begin{array}{cc} 1& 0\\ 0& -1\end{array}\right)R=S$.

To any $j\in\{0,\ldots,q\}$ we associate :
\begin{itemize}
\item $g_j:={\rm diag}(\underbrace{1,\ldots,1}_{p-j\ {\rm times}},\underbrace{R,\ldots,R}_{j\ {\rm times}},\underbrace{1,\ldots,1}_{q-j\ {\rm times}})\in G$,

\item the permutation $w_j\in \mathfrak{S}_n$ that fixes the elements of $[1,\cdots, p-j]$\break $\cup[p+j+1,\cdots,n]$ and such that
$$
w_j(p-j+2k-1)=p-j+k,\quad w_j(p-j+2k)=p+k,\quad {\rm for}\quad 1\leq k\leq j,
$$

\item $k_j= \tau_j g_j\in G$, where $\tau_j=\tau(w_j)\in N(T)$,

\item $x_j=k_j T\in \Fcal$.
\end{itemize}

The adjoint map $Ad(\tau_j):G\to G$ sends the matrix 
${\rm diag}(a_1,\ldots,a_{p-j},b_1,\ldots,b_{2j},c_1,\ldots,c_{q-j})$ to the matrix
${\rm diag}(a_1,\ldots,a_{p-j},b_1,b_3,\ldots,b_{2j-1},b_2,b_4,\ldots,b_{2j},c_1,\ldots,c_{q-j})$.
 
We see then that 
$$
\sigma_j:=k_j^{-1}\Delta k_j=
{\rm diag}(\underbrace{1,\ldots,1}_{p-j\ {\rm times}},\underbrace{S,\ldots,S}_{j\ {\rm times}},\underbrace{-1,\ldots,-1}_{q-j\ {\rm times}})
$$ 
and $k_j^{-1}\theta(k_j)=\sigma_j \Delta$ belong to $N(T)$. Thus the elements $x_0,\ldots,x_q$ belongs to $Z_\theta$.

\begin{lem}
In the flag manifold $\Fcal$ the set $Z_\theta$ has the following description:
$$
Z_\theta=\bigcup_{0\leq j\leq q} \ \bigcup_{\bar{w}\in W_{x_j}\backslash W} H x_j w
$$
So we have  $H\backslash Z_{\theta}\slash W=\{\bar{x}_0,\ldots,\bar{x}_q\}$.
\end{lem}
\begin{proof} If $1\leq a<b\leq n$, we denote by $\tau_{a,b}\in N(T)$ the permutation matrix associated to the transposition $(a,b)$.

Let $gT\in Z_{\theta}$. Then $k:=g^{-1}\theta(g)\Delta=g^{-1}\Delta g$ is an element of order two in $N(T)$. 
The Weyl group element $\bar{k}\in W$ is of order two, 
then there exists $0\leq l\leq n/2$, and a family $(a_1<b_1),\ldots,(a_l<b_l)$ of disjoint couples in $\{1,\ldots,n\}$  such that 
$k T= \tau_{a_1,b_1}\ldots \tau_{a_l,b_l}T$. 

Now, if we use the fact that the characteristic polynomial of $k\in G$ is equal to $(X-1)^p(X+1)^q$ with  $p\geq q\geq 1$, we see that 

$\bullet $ $l\leq q$,

$\bullet$ there exists $n\in N(T)$ such that $nkn^{-1}= \sigma_l=k_l^{-1} \Delta k_l$.

If we take $w=\bar{n}\in W$, the previous identity says that $g\in Hk_l w T$.  
\end{proof}

\medskip

\subsubsection{Localized indices}

We work with the groups $T\subset H=U(p)\times U(q)\subset G=U(n)$ and the corresponding Lie algebras 
$\tgot\subset \hgot\subset \ggot$. Let $\Rgot=\{\varepsilon_r-\varepsilon_s\}$ be the set of non-zero roots for 
the action of $T$ on $\ggot\otimes\Cbb$. 
We choose the Weyl chamber so that $\Rgot^+:=\{\varepsilon_r-\varepsilon_s,\ 1\leq r<s\leq n\}$.

Let $j\in\{0,\ldots,q\}$. The aim of this section is to compute the localized index $Q_{\bar{x}_j}(\lambda)\in \wR(H)$. 
In order to have a fairly simple expression we will rewrite the terms of the form 
${\rm Ind}^{H}_{H_{x_j}}(\Cbb_{\beta}\otimes \bigwedge 
\mathbb{E}^{\hbox{\scriptsize\rm  ci}}_{x_j})$.

Let $\{1,\ldots, n\}=I^1_j\cup I^2_j\cup I^3_j\cup I^4_j$ where 
$I^1_j=\{1\leq k\leq p-j\}$, $I^2_j=\{p-j+1\leq k\leq p\}$,
$I^3_j=\{p+1\leq k\leq p+j\}$, and $I^4_j=\{p+j+1\leq k\leq n\}$.

For the maximal torus $T\subset G$ we have a decomposition
$$
T\simeq T^1_{j}\times T^2_{j}\times T^3_{j}\times T^4_{j}
$$ 
where $T^p_{j}=\{(t_k)_{k=1}^n, t_k\in U(1), t_k=1\ {\rm unless}\ k\in I^p_j\}$. Let $T_{j} \subset T^2_{j}\times T^3_{j}$ be the subtorus defined by the relations: an element 
$((t_k)_{k=1}^n,(s_k)_{k=1}^n) \in T_j^2\times T_j^3$ belongs to $T_j$ 
if and only if $t_{p-j+k}=s_{p+k}$ for all $1\leq k\leq j$.

\medskip

The elements of order two $\sigma_j\in G$ induce involutions on $G$ (by conjugation) that we still denote by $\sigma_j$. We start with a basic lemma whose proof is left to the reader.

\begin{lem} \label{lem:calcul_x_j}
Let $x_j=k_j T\in \Fcal$.

$\bullet$ The adjoint map $Ad(k_j):\ggot\to\ggot$ realizes an isomorphism between the vector space $\tgot$ equipped with the involution induced by $\sigma_j$ and  
the vector space $\ggot_{x_j}$ equipped with the involution $\theta$. 

$\bullet$ The group $N(T)^{\sigma_j}/T^{\sigma_j}$ is isomorphic with 
$\mathfrak{S}_{p-j}\times\mathfrak{S}_{q-j}\times\mathfrak{S}_{j}\times \{\pm\}^j$.

$\bullet$ The adjoint map $Ad(k_j):G\to G$ induces an isomorphism $N(T)^{\sigma_j}/T^{\sigma_j}$ $\simeq W_{x_j}$.

$\bullet$ The stabilizer subgroup $H_{x_j}$ is equal to $T^1_{j}\times T_{j}\times T_{j}^4 \subset T$.

$\bullet$ If $\Cbb_\alpha$ is a character of $T$, then $\Cbb_{k_j\alpha}$ is a character of $G_{x_j}$ and $\Cbb_{\tau_j\alpha}$ is  a character of 
$T$. We have the relation
$$
\Cbb_{k_j\alpha}\vert_{H_{x_j}}=\Cbb_{\tau_j\alpha}\vert_{H_{x_j}}.
$$

$\bullet$ The set of roots $\Rgot_{x_j}^{\hbox{\scriptsize \rm ci}}$ is equal to 
$$
 k_j\cdot\left\{\varepsilon_r-\varepsilon_s, 1\leq r<s\leq p-j\right\}
\bigcup k_j\cdot\left\{\varepsilon_r-\varepsilon_s, p+j+1\leq r<s\leq n\right\}
$$
and $\Rgot_{x_j}^{\hbox{\scriptsize \rm nci}}= k_j\cdot\left\{\varepsilon_r-\varepsilon_s, 1\leq r\leq p-j\ \&\ p+j+1\leq s\leq n \right\}$.

\end{lem}

We denote by $\mathbb{M}_j$ the $T$-module $\Cbb^{p-j}\otimes (\Cbb^{q-j})^*$ where the subgroup $T^2_{j}\times T^3_{j}$ acts trivially and 
the $T^1_{j}\times T^4_{j}$-action is the canonical one. Thanks to Lemma \ref{lem:calcul_x_j}, we have the following isomorphisms of $H_{x_j}$-modules:
$\mathbb{E}^{\hbox{\scriptsize \rm nci}}_{x_j}\simeq \mathbb{M}_j$. Following Lemma \ref{lem:polarized-module}, one can associate the modules 
$(\mathbb{M}_j)^\pm_w$ and $\vert\mathbb{M}_j\vert_w$ to each $w\in W$.

We consider the Lie group 
$$
K_j:=U(p-j)\times U(q-j)
$$
that we view as a subgroup of $H$ in such a way that $T^1_{j}\times T^4_{j}$ is a maximal torus of $K_j$. 
A set of positive roots for $(K_j,T^1_{j}\times T^4_{j})$ is 
$\varepsilon_r-\varepsilon_s$ for $1\leq r<s\leq p-j$ and $p+j+1\leq r<s\leq n $. We equip $\kgot_j\slash\left[\tgot^1_j\times \tgot^4_j\right]$ 
with a complex structure 
such that 
$$
\mathbb{E}^{\hbox{\scriptsize \rm ci}}_{x_j}\simeq \kgot_j\slash 
\left[\tgot^1_j\times \tgot^4_j\right]
$$
is an isomorphism of $T_j^1\times T_j^4$-modules.

The holomorphic induction map ${\rm Hol}^{K_j}_{T^1_{j}\times T^4_{j}}: \wR(T^1_{j}\times T^4_{j})\to \wR(K_j)$ is defined as follows:
$$
{\rm Hol}^{K_j}_{T^1_{j}\times T^4_{j}}(V):={\rm Ind}^{K_j}_{T^1_{j}\times T^4_{j}}(V\otimes \bigwedge\kgot_j\slash\left[\tgot^1_j\times \tgot^4_j\right]).
$$
If $a=(a_1\geq \cdots\geq a_{p-j})\in \Zbb^{p-j}$ and 
$b= (b_1\geq \cdots\geq b_{q-j})\in \Zbb^{q-j}$, then $\Cbb_{(a,b)}$ defines a character of $T^1_{j}\times T^4_{j}$ and 
$$
{\rm Hol}^{K_j}_{T^1_{j}\times T^4_{j}}\left(\Cbb_{(a,b)}\right)= V_a^{U(p-j)}\otimes V_b^{U(q-j)}
$$
is the irreducible representation of $K_j$ with highest weight $(a,b)$.

A character $\Cbb_{\beta}$ of the torus $T$ can be written 
$\Cbb_{\beta}= \Cbb_{\beta^{14}}\otimes \Cbb_{\beta^{23}}$ where $\Cbb_{\beta^{14}}$ is a character of 
$T^1_j\times T_j^4$ and $\Cbb_{\beta^{23}}$ is a character of $T^2_j\times T_j^3$. Note that  
$\Cbb_{\tau_j\beta}\vert_{H_{x_j}}= \Cbb_{\beta^{14}}\otimes \Cbb_{\beta'}$ where 
$\beta'=\tau_j\beta^{23}$ defines a character of $T_j\subset T^2_j\times T_j^3$.

\begin{lem}
Let $\Cbb_{\beta}$ be a character of 
$T$. Then ${\rm Ind}^{H}_{H_{x_j}}(\Cbb_{\beta}\vert_{H_{x_j}}\otimes \bigwedge 
\mathbb{E}^{\hbox{\scriptsize\rm  ci}}_{x_j})$ is equal to 
$$
{\rm Ind}^{H}_{K_j\times T^2_{j}\times T^3_{j}}\left( {\rm Hol}^{K_j}_{T^1_{j}\times T^4_{j}}\left(\Cbb_{\beta^{14}}\right)\otimes \Cbb_{\beta^{23}}\otimes 
L^2(\left[T^2_{j}\times T^3_{j}\right]/T_j)\right),
$$
where $L^2(\left[T^2_{j}\times T^3_{j}\right]/T_j)={\rm Ind}^{T^2_{j}\times T^3_{j}}_{T_j}(1)\in \widehat{R}(T^2_{j}\times T^3_{j})$.
\end{lem}

\begin{rem}To gain some space in our formulas, we will write ${\rm Hol}^{K_j}_{T^1_{j}\times T^4_{j}}\left(\Cbb_{\beta}\right)$ instead of 
${\rm Hol}^{K_j}_{T^1_{j}\times T^4_{j}}\left(\Cbb_{\beta^{14}}\right)\otimes \Cbb_{\beta^{23}}$
\end{rem}

We need to fix some notations.

\begin{defi} 
$\bullet$ Let $\chi: H\to\Cbb$ be the character $(A,B)\mapsto \det(A)\det(B)^{-1}$. 

$\bullet$ Let $\psi_j$ be the character\footnote{Remark that $\psi_j$ is trivial $T^1_{j}\times T^3_{j}\times T^4_{j}$.}  of $T$ associated to the weight 
$$
\sum_{1\leq k\leq j} (q-p+2+2j-4k)\varepsilon_{p-j+k}.
$$

$\bullet$ For any $(j,w)\in [0,q]\times W$, we define the integer $d_{j,w}$ by the relation
$$
d_{j,w}=\dim (\mathbb{M}_j)_w^+ + | \{1\leq k\leq j,\ w^{-1}(p-j+2k-1)<w^{-1}(p-j+2k)\} |.
$$
\end{defi}

A small computation gives the following lemma.

\begin{lem}
\begin{itemize}
\item The $H_{x_j}$-character $\Cbb_{\delta(x_j)}$ is equal to 
$\chi^{\otimes j}\otimes \psi_j   \vert_{H_{x_j}}$.

\item For any $(j,w)\in [0,q]\times W$, we have 
$(-1)^{n_{x_j}+k_{x_j,w}}= (-1)^{j(n+1)}(-1)^w (-1)^{d_{j,w}}$.
\end{itemize}
\end{lem}

The main result of this section is the following proposition.

\begin{prop}
$$
V_{\lambda}^{U(n)}\vert_{U(p)\times U(q)}=\sum_{j=0}^q Q_{\bar{x}_j}(\lambda)
$$
where $Q_{\bar{x}_j}(\lambda)\in \widehat{R}(U(p)\times U(q))$ is determined by the relation
$$
Q_{\bar{x}_j}(\lambda)=\frac{(-1)^{j(n+1)}}{|W_{x_j}|}\ \chi^{\otimes j} \otimes \sum_{w\in W}(-1)^w(-1)^{d_{j,w}} \ 
{\rm Ind}^{U(p)\times U(q)}_{K_j\times T^2_{j}\times T^3_{j}}\left(\mathbb{A}_{j}^w(\lambda)\otimes \psi_j\right). 
$$
Here the elements $\mathbb{A}_{j}^w(\lambda)\in \widehat{R}(K_j\times T^2_{j}\times T^3_{j})$ are defined as follows:
$$
\mathbb{A}_{j}^w(\lambda)={\rm Hol}^{K_j}_{T^1_{j}\times T^4_{j}}\left(\Cbb_{\tau_j(w\bullet\lambda)}
\otimes \det((\mathbb{M}_{j})^+_w)\otimes {\rm Sym}(|\mathbb{M}_{j}|_w)\right)\otimes L^2(\left[T^2_{j}\times T^3_{j}\right]/T_j).
$$
\end{prop}

We finish this section by considering particular situations.

\subsubsection{The extreme cases : $j=0$ or $j=q$}

When $j=0$, the torus $T_0^2$ and $T_0^3$ are trivial and $K_0=U(p)\times U(q)=H$. Moreover 
$\mathbb{M}_0=\Cbb^{p}\otimes(\Cbb^q)^*$ and $d_{0,w}=\dim (\mathbb{M}_0)_w^+$. Thanks to Lemma \ref{lem:calcul_x_j}, we know also that 
$W_{x_0}\simeq \mathfrak{S}_p\times \mathfrak{S}_q$.

So we get the formula
$$
Q_{\bar{x}_0}(\lambda)=\frac{1}{p!q!}
\sum_{w\in W}(\pm)_w\ 
{\rm Hol}^{H}_{T}\left(\Cbb_{w\bullet\lambda}
\otimes \det((\mathbb{M}_{0})^+_w)\otimes {\rm Sym}(|\mathbb{M}_{0}|_w)\right)
$$
where $(\pm)_w= (-1)^w(-1)^{\dim (\mathbb{M}_0)_w^+}$.

\begin{rem}An useful exercise is to consider the term
$$\mathbb{A}_w:=(\pm)_w\ 
{\rm Hol}^{H}_{T}\left(\Cbb_{w\bullet\lambda}
\otimes \det((\mathbb{M}_{0})^+_w)\otimes {\rm Sym}(|\mathbb{M}_{0}|_w)\right)$$ and  verify that 
$\mathbb{A}_{w'w}=\mathbb{A}_w$ when $w'\in W_{x_0}$.
\end{rem}

\medskip

When $j=q$, the torus $T_q^4$ is trivial, $K_q=U(p-q)$ and $\mathbb{M}_q=\{0\}$. Moreover  $W_{x_q}\simeq \mathfrak{S}_{p-q}
\times \mathfrak{S}_q\times \{\pm\}^q$. In this case we obtain
$$
Q_{\bar{x}_q}(\lambda)= \frac{(-1)^{q(n+1)}}{(p-q)!q! 2^q}\ \chi^{\otimes q} \otimes\sum_{w\in W}(-1)^w(-1)^{d_{q,w}} \  Q_q^w(\lambda)
$$
with 
$$
Q_q^w(\lambda)={\rm Ind}^{U(p)\times U(q)}_{U(p-q)\times T^2_{q}\times T^3_{q}}
\left({\rm Hol}^{U(p-q)}_{T^1_{q}}\left(\Cbb_{\tau_q(w\bullet\lambda)}\right)\otimes\psi_q \otimes L^2(\left[T^2_{q}\times T^3_{q}\right]/T_q)\right). 
$$

\subsubsection{$U(n-1)\times U(1)\subset U(n)$}

Here we are in the case where $q=1$, and so 
$$
V_{\lambda}^{U(n)}\vert_{U(n-1)\times U(1)}=Q_{\bar{x}_0}(\lambda)+ Q_{\bar{x}_1}(\lambda).
$$

To simplify the expression of $Q_{\bar{x}_0}(\lambda)$ we use the fact that the quotient $W_{x_0}\backslash W$ is represented by the class of the 
elements $\tau_{k,n}\in G$ associated to the transposition $(k,n)$ for $1\leq k\leq n$. We write $T=T'\times U(1)$ where $T'$ is a maximal torus of $U(n-1)$. The $T'$-module 
$\Cbb^{n-1}$ can be decomposed as $\mathbb{V}_k\oplus\mathbb{V}'_k$ where $\mathbb{V}_k=\sum_{j=1}^{k-1}\Cbb_{\varepsilon_j}$ and 
$\mathbb{V}'_k=\sum_{j=k}^{n-1}\Cbb_{\varepsilon_j}$.

The $T$-module $\mathbb{M}_0$ is equal to $\Cbb^{n-1}\otimes \Cbb^*=\mathbb{V}_k\otimes \Cbb_{-\varepsilon_n}\oplus 
\mathbb{V}'_k\otimes \Cbb_{-\varepsilon_n}$ and the polarized $T$-module $|\mathbb{M}_{0}|_{\tau_{k,n}}$ is equal to 
$\mathbb{V}_k\otimes \Cbb_{-\varepsilon_n}\oplus 
\overline{\mathbb{V}'_k}\otimes \Cbb_{\varepsilon_n}$. 
We have $\dim (\mathbb{M}_{0})^+_{\tau_{k,n}}= k-1$ and $\det(\mathbb{M}_{0})^+_{\tau_{k,n}}= \Cbb_{\mu_k}\otimes \Cbb_{\varepsilon_n}^{\otimes 1-k}$ with 
$\mu_k=\sum_{j=1}^{k-1}\varepsilon_j$.

So we obtain
\begin{eqnarray*}
   \lefteqn{Q_{\bar{x}_0}(\lambda)=} \\
   &&\sum_{\stackrel{a,b\geq 0}{1\leq k\leq n}}\ (\pm)_k\ 
{\rm Hol}^{U(n-1)}_{T'}\left(\Cbb_{\tau_{k,n}\bullet\lambda+\mu_k}\otimes {\rm Sym}^b(\mathbb{V}_k)\otimes {\rm Sym}^a(\overline{\mathbb{V}'_k})
\right)\otimes \Cbb_{\varepsilon_n}^{\otimes 1+a-b-k},\nonumber
\end{eqnarray*}
where $(\pm)_k=(-1)^k$ if $k<n$ and $(\pm)_n=(-1)^{n-1}$.

We consider now the term $Q_{\bar{x}_1}(\lambda)$. When $j=q=1$, the torus $T_1^4$ is trivial, $K_1=U(n-2)$ and $\mathbb{M}_1=\{0\}$. Moreover  
$W_{x_1}\simeq \mathfrak{S}_{n-2}\times \{\pm\}$, $\tau_1=Id$ and $\psi_1=(2-n)\varepsilon_{n-1}$. Here the quotient $W_{x_1}\backslash W$ is 
represented by the class of the elements $\tau_{l,n}\tau_{k,n-1}$ for $1\leq k<l\leq n$. We denote by $\lambda_{kl}$ the term $\tau_{l,n}\tau_{k,n-1}\bullet\lambda$.

In this case we obtain
$$
Q_{\bar{x}_1}(\lambda)=(-1)^{n}\ \chi\,  \otimes\left( Q_{\bar{x}_1}^{n-1,n}(\lambda)-\sum_{1\leq k< n-1} Q_{\bar{x}_1}^{k,n}(\lambda)+
\sum_{1\leq k<l\leq n-1}  Q_{\bar{x}_1}^{k,l}(\lambda)\right)
$$
with 
\begin{eqnarray*}
Q_{\bar{x}_1}^{k,l}(\lambda)&=&{\rm Ind}^{U(n-1)\times T^3_{1}}_{U(n-2)\times T^2_{1}
\times T^3_{1}}
\left({\rm Hol}^{U(n-2)}_{T^1_{1}}\left(\Cbb_{\lambda_{kl}}\right)\otimes\psi_1 \otimes L^2(\left[T^2_{1}\times T^3_{1}\right]/T_1)\right)\\
&=& \sum_{a\in\Zbb} {\rm Ind}^{U(n-1)}_{U(n-2)\times T^2_{1}}
\left({\rm Hol}^{U(n-2)}_{T^1_{1}}\left(\Cbb_{\lambda_{kl}}\right) \otimes 
\Cbb_{\varepsilon_{n-1}}^{\otimes a}\right)
\otimes \Cbb_{\varepsilon_{n}}^{\otimes 2-n-a}.
\end{eqnarray*}

\medskip

Let us finish this section by considering the simplest example: $U(1)\times U(1)\subset U(2)$. Take $\lambda=(\lambda_1\geq \lambda_2)\in\widehat{U(2)}$. We have $V_{\lambda}^{U(2)}\vert_{U(1)\times U(1)}=Q_{\bar{x}_0}(\lambda)+ Q_{\bar{x}_1}(\lambda)$ where 
$$
Q_{\bar{x}_0}(\lambda)=-\ \Cbb_\lambda\otimes \sum_{-\infty}^{\lambda_2-\lambda_1-1}\Cbb_{\varepsilon_1-\varepsilon_2}^{\otimes k} 
\ -\ \Cbb_\lambda\otimes  \sum_{k\geq 1}\Cbb_{\varepsilon_1-\varepsilon_2}^{\otimes k} 
$$
and $Q_{\bar{x}_1}(\lambda)=\Cbb_\lambda\otimes \sum_{k\in\mathbb{Z}}\Cbb_{\varepsilon_1-\varepsilon_2}^{\otimes k}$. We recover the basic relation
$$
V_{\lambda}^{U(2)}\vert_{U(1)\times U(1)}=\Cbb_\lambda\otimes \sum_{k=\lambda_2-\lambda_1}^0\Cbb_{\varepsilon_1-\varepsilon_2}^{\otimes k} .
$$

\subsubsection{$U(n-1)\subset U(n)$}\label{sec:U(n)}

If we restrict the representation $V_{\lambda}^{U(n)}$ to the subgroup \break $U(n-1)$, we get 
\begin{equation}\label{eq:U-n-U-n-1}
V_{\lambda}^{U(n)}\vert_{U(n-1)}=Q_{0}(\lambda)+ Q_{1}(\lambda),
\end{equation}
where the characters $Q_{0}(\lambda),Q_{1}(\lambda)\in\wR(U(n-1))$ are given by the relations  
$$
Q_{0}(\lambda)=
\sum_{k=1}^{n}\ (\pm)_k\ 
{\rm Hol}^{U(n-1)}_{T'}\left(\Cbb_{\tau_{k,n}\bullet\lambda+\mu_k}\otimes {\rm Sym}(\mathbb{V}_k)\otimes {\rm Sym}(\overline{\mathbb{V}'_k})
\right),
$$
and 
$$
Q_{1}(\lambda)=(-1)^n{\rm det}\,  \otimes\left( Q_{1}^{n-1,n}(\lambda)-\sum_{1\leq k< n-1} Q_{1}^{k,n}(\lambda)+
\sum_{1\leq k<l\leq n-1}  Q_{1}^{k,l}(\lambda)\right),
$$
with $Q_{1}^{k,l}(\lambda)= {\rm Ind}^{U(n-1)}_{U(n-2)}
\left({\rm Hol}^{U(n-2)}_{T^1_{1}}\left(\Cbb_{\lambda_{kl}}\right)\right)$.

Let's detail expression (\ref{eq:U-n-U-n-1}) when $n=3$.

Small calculations give 
$Q_{0}(\lambda)=\mathbb{B}_1(\lambda)+\mathbb{B}_2(\lambda)+\mathbb{B}_3(\lambda)$
with 
\begin{eqnarray*}
   \mathbb{B}_1(\lambda)&=&{\rm Hol}^{U(2)}_{T}\left(\Cbb_{\tau_{1,3}\bullet\lambda+\mu_1}\right)\otimes{\rm Sym}(\overline{\Cbb^2})
= \sum_{\lambda_2-1\geq a\geq \lambda_3-1\geq b} V^{U(2)}_{(a,b)},\\
\mathbb{B}_2(\lambda)&=&{\rm Hol}^{U(2)}_{T}\left(\Cbb_{\tau_{2,3}\bullet\lambda+\mu_2}\otimes{\rm Sym}(\Cbb_{\epsilon_1})
 \otimes{\rm Sym}(\overline{\Cbb_{\epsilon_2}})
 \right)= \sum_{a\geq \lambda_1+1,\, \lambda_3-1\geq b} V^{U(2)}_{(a,b)},\\
  \mathbb{B}_2(\lambda)&=&{\rm Hol}^{U(2)}_{T}\left(\Cbb_{\tau_{3,3}\bullet\lambda+\mu_3}\right)\otimes{\rm Sym}(\Cbb^2)
 =  \sum_{a\geq \lambda_1+1\geq b\geq \lambda_2+1} V^{U(2)}_{(a,b)}.
\end{eqnarray*}

For the other term, we obtain $Q_{1}(\lambda)=\mathbb{A}_1(\lambda) - \mathbb{A}_2(\lambda)- \mathbb{A}_3(\lambda)$ with
\begin{eqnarray*}
  \mathbb{A}_1(\lambda)&=&\det\otimes\, Q_{1}^{1,3}(\lambda)=\det\otimes\,  {\rm Ind}^{U(2)}_{U(1)}
\left(\Cbb_{\lambda_{2}-1}\right)
= \sum_{a\geq \lambda_2\geq b} V^{U(2)}_{(a,b)},\\
\mathbb{A}_2(\lambda)&=&\det\otimes\, Q_{1}^{2,3}(\lambda)=\det\otimes\,  {\rm Ind}^{U(2)}_{U(1)}
\left(\Cbb_{\lambda_{1}}\right)
= \sum_{a\geq \lambda_1+1\geq b} V^{U(2)}_{(a,b)},\\
\mathbb{A}_3(\lambda)&=&\det\otimes\, Q_{1}^{1,2}(\lambda)=\det\otimes\,  {\rm Ind}^{U(2)}_{U(1)}
\left(\Cbb_{\lambda_{3}-2}\right)
=\sum_{a\geq \lambda_3-1\geq b} V^{U(2)}_{(a,b)} .
\end{eqnarray*}

Finally one checks that the decomposition 
$$
V_\lambda^{U(3)}\vert_{U(2)}=\mathbb{A}_1(\lambda)-\mathbb{A}_2(\lambda)-\mathbb{A}_3(\lambda)+ \mathbb{B}_1(\lambda) + \mathbb{B}_2(\lambda)+ \mathbb{B}_3(\lambda)
$$
permits to recover the classical relation $V_\lambda^{U(3)}\vert_{U(2)}=\sum_{\lambda_1\geq a\geq \lambda_2\geq b\geq\lambda_3} V^{U(2)}_{(a,b)}$ (see \cite{Goodman-Wallach}).

\begin{figure}[h]
    \centering
   \includegraphics[width=14cm]{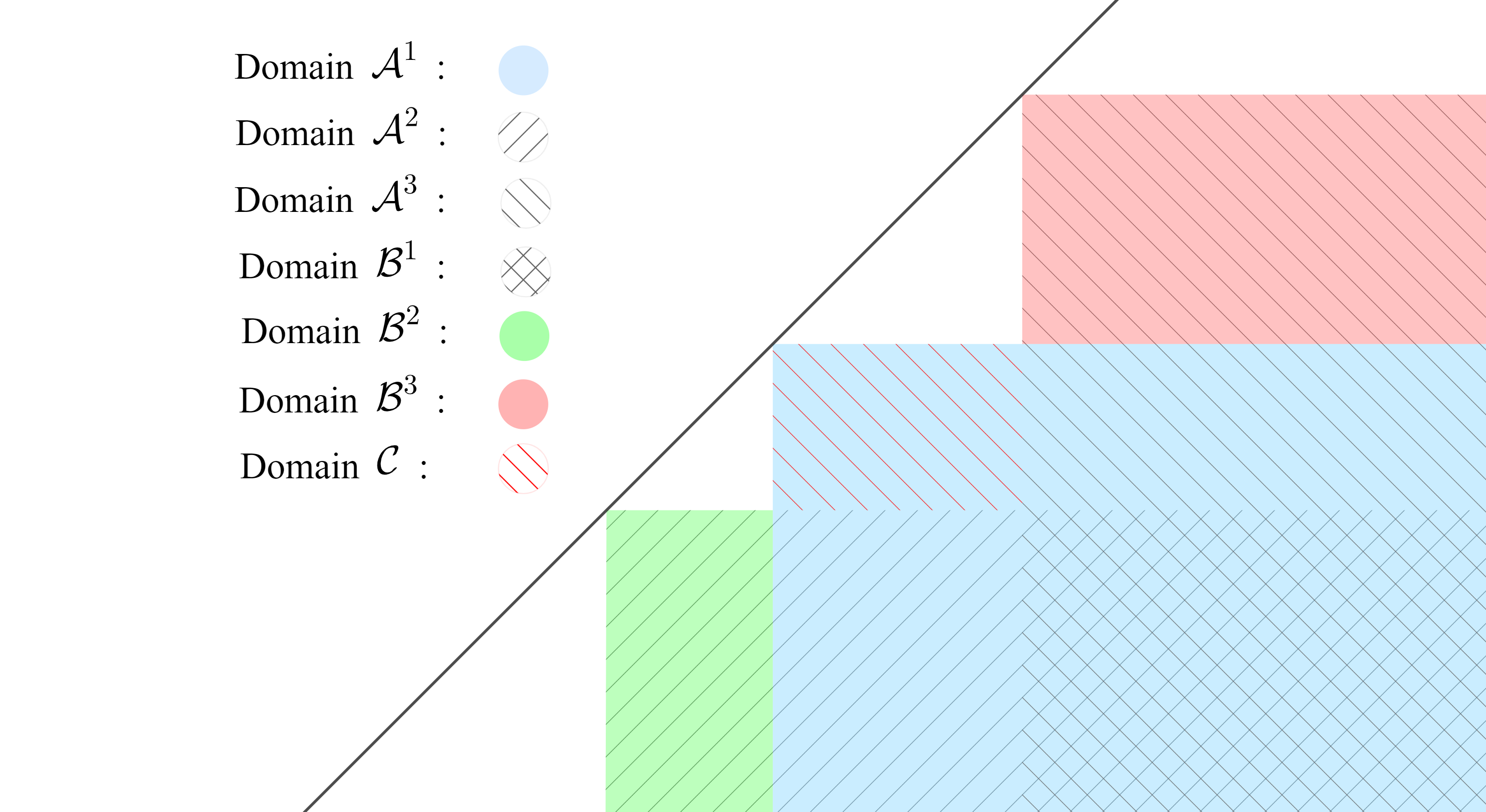} 
    \caption{Restriction from $U(3)$ to $U(2)$}
    \label{fig:U(3)-pep}
\end{figure}

In Figure \ref{fig:U(3)-pep}, we can visualise the supports of the differents characters: we have 

$$
\mathbb{A}_1(\lambda)=\sum_{\mu\in\mathcal{A}_1}V^{U(2)}_\mu,
\quad 
\mathbb{A}_2(\lambda)=\sum_{\mu\in\mathcal{A}_2\cup\mathcal{B}_1}V^{U(2)}_\mu,
\quad
\mathbb{A}_3(\lambda)=\sum_{\mu\in\mathcal{A}_3\cup\mathcal{B}_1}V^{U(2)}_\mu,\quad
$$

$$
\mathbb{B}_1(\lambda)=\sum_{\mu\in\mathcal{B}_1}V^{U(2)}_\mu,
\quad 
\mathbb{B}_2(\lambda)=\sum_{\mu\in\mathcal{B}_2}V^{U(2)}_\mu,
\quad
\mathbb{B}_3(\lambda)=\sum_{\mu\in\mathcal{B}_3}V^{U(2)}_\mu,\quad
$$
so that $V_\lambda^{U(3)}\vert_{U(2)}=\sum_{\mu\in\mathcal{C}}V^{U(2)}_\mu$.

\section{Kostant multiplicity formula}\label{sec:Kostant}

The aim of this section is first to recall the Kostant multiplicity formula : we follow the line of \cite{Goodman-Wallach}, Section 8.2. Then, we rewrite it in a form similar to the one we use in this article (see Proposition \ref{prop:kostant-formula}). Finally, we detail Kostant's multiplicity formula for the restriction of $U(n)$ to $U(n-1)$, in order to compare it with 
the calculations done in Section \ref{sec:U(n)}.

Let $G'\subset G$ be two connected compact Lie groups with maximal tori 
$T'\subset T$. The corresponding Lie algebras are $\tgot\subset\ggot$ and $\tgot'\subset\ggot'$. In this section, we make the following regularity assumption: 

\medskip 

$(R)$ \quad The centralizer $Z_\ggot(\tgot')$ of $\tgot'$ in $\ggot$ is abelian.

\medskip

We recall the following well-known fact.

\begin{lem}
The assumption $(R)$ is valid when $G'$ is the connected component of a fixed-point subgroup of an involution.
\end{lem}

{\em Proof :} Suppose that $G'=(G^\tau)_0$ for some involution $\tau$. Then, we have a decomposition $\ggot=\ggot'\oplus \qgot$  where $\qgot=\{X\in\ggot,\tau(X)=-X\}$. The centralizer 
$Z_\ggot(\tgot')$  is stable under the involution $\tau$ and under the adjoint action of $T$. Thus  $\tgot\subset Z_\ggot(\tgot')=\tgot'\oplus Z_\qgot(\tgot')$:  
in particular the torus $T$ is invariant under $\tau$.

If $Z_\ggot(\tgot')$ is not abelian, there exits a roots $\alpha\in\Rgot$ such that $(\ggot_\Cbb)_\alpha\subset  Z_\ggot(\tgot')_\Cbb=\tgot'_\Cbb\oplus Z_\qgot(\tgot')_\Cbb$. Then we obtain a contradiction: on one hand $(\ggot_\Cbb)_\alpha\subset  Z_\ggot(\tgot')_\Cbb$ implies that $\alpha\vert_{\tgot'}=0$ and on the other hand since  
$(\ggot_\Cbb)_\alpha\subset \qgot_\Cbb$, we must have $\sigma(\alpha)=\alpha$. The two conditions $\alpha\vert_{\tgot'}=0$ and $\sigma(\alpha)=\alpha$ implies that $\alpha=0$. $\Box$.

\medskip

Let $\Rgot$ and $\Rgot'$ and be the set of roots for the pairs $T\subset G$ and $T'\subset G'$. Note that assumption $(R)$ is equivalent to : 

\medskip 

$(R')$ \quad There exists $X_o\in \tgot'$ such that $\langle \alpha, X_o\rangle \neq 0$ for all $\alpha\in\Rgot$.

\medskip

If $\xi\in\tgot^*$ we write $\overline{\xi}$ for the restriction of $\xi$ to $(\tgot')^*$. Because of our assumption, $\overline{\alpha}\neq 0$ for all $\alpha\in\Rgot$.

The positive roots are  $\Rgot_+:=\{\alpha\in\Rgot, \langle\alpha,X_o\rangle>0\}$ and $\Rgot'_+:=\{\beta\in\Rgot', \langle\beta,X_o\rangle>0\}$. 
We write $\overline{\Rgot_+}:=\{\overline{\alpha}, \alpha\in\Rgot_+\}$ for the set of positive restricted roots: we keep track of the multiplicity $n_a=\#\{\alpha\in\Rgot_+,\overline{\alpha}=a\}$ of each element $a\in\overline{\Rgot_+}$.

Since $\Rgot'_+$ is contained in $\overline{\Rgot_+}$, we may consider the set of roots $\Sigma:=\overline{\Rgot_+}-\Rgot'_+$: the multiplicity of $\beta\in\overline{\Rgot_+}$ in $\Sigma$ is equal to 
$$
m_\beta:=
\begin{cases}
n_\beta\hspace{11mm} {\rm if}\quad \beta \notin \Rgot'_+,\\
n_\beta-1\quad {\rm if}\quad \beta \in \Rgot'_+.
\end{cases}
$$

Let $\Lambda'\subset(\tgot')^*$ be the lattice of weights for the torus $T'$. Let $(\tgot')_{+}^*$ be the Weyl chamber associated to the system $\Rgot'_+$. The irreducible representations of $G'$ are parameterized by $\Lambda'_+=\Lambda'\cap (\tgot')_+^*$.

\begin{defi} We denote by $\Pcal_\Sigma: \Lambda'\to \Nbb$ the partition function associated to the set $\Sigma$. For all $\xi'\in\Lambda'$, $\Pcal_\Sigma(\xi')$ is the number of way of writing $\xi'= \sum_{\beta\in\Sigma} x_\beta \beta$, where $x_\beta\in\Nbb$ and each $\beta$ that occurs is counted with multiplicity $m_\beta$.
\end{defi}

For dominant weights $\lambda\in\Lambda_+$ and $\mu\in\Lambda'_+$, we denote by 
$m_\lambda(\mu)$ the multiplicity of the irreducible $G'$-representation $V_\mu^{G'}$ with highest weight $\mu$ in  the irreducible $G$-representation $V_\lambda^G$ with highest weight $\lambda$.

If $w\in W$, we note $w\bullet\lambda:=w(\lambda+\rho)-\rho$ where $\rho$ is the half sum of the positive roots.

\begin{theo}
The branching multiplicities are
\begin{equation}\label{eq:kostant-1}
m_\lambda(\mu)=\sum_{w\in W}(-1)^w\, 
\Pcal_\Sigma(\overline{w\bullet\lambda}-\mu).
\end{equation}  
\end{theo}

We briefly recall  how to obtain (\ref{eq:kostant-1}). Let $\chi_\lambda^G$ be the character of $V_\lambda^G$. The Weyl relation gives 
$$
\chi_\lambda^G\vert_T \prod_{\alpha\in\Rgot_+}(1-e^{-\alpha})=\sum_{w\in W}(-1)^w \, e^{w\bullet \lambda}.
$$
When we restrict this identity to $T'\subset T$, the relations
$$
\prod_{\alpha\in\Rgot_+}(1-e^{-\alpha})\vert_{T'}=
\prod_{\beta\in\Rgot'_+}(1-e^{-\beta})
\prod_{\gamma\in\Sigma}(1-e^{-\gamma})
$$
and
$$
\prod_{\gamma\in\Sigma}(1-e^{-\gamma})\left(\sum_{\xi'\in\Lambda'}
\Pcal_\Sigma(\xi')\, e^{-\xi'}\right)=1
$$
permit to obtain 
\begin{eqnarray*}
\chi_\lambda^G\vert_{T'}\prod_{\beta\in\Rgot'_+}(1-e^{-\beta})
&=&\left(\sum_{w\in W}(-1)^w \, e^{\overline{w\bullet \lambda}}\right)
\left(\sum_{\xi'\in\Lambda'}
\Pcal_\Sigma(\xi')\, e^{-\xi'}\right)  \\
&=& \sum_{\xi'\in\Lambda'} N_\lambda(\xi') e^{\xi'}.
\end{eqnarray*}
with $N_\lambda(\xi'):=\sum_{w\in W}(-1)^w \, \Pcal_\Sigma(\overline{w\bullet \lambda}-\xi')$.

On the other hand, we have the decomposition
$\chi_\lambda^G\vert_{H}=\sum_{\mu\in\Lambda'_+}m_\lambda(\mu)\chi_\mu^{G'}$ and then\footnote{Here $w'\bullet\xi':=w'(\lambda+\rho')-\rho'$ where $\rho'=\frac{1}{2}\sum_{\beta\in\Rgot'_+}\beta$.} 
\begin{eqnarray*}
\chi_\lambda^G\vert_{T'}\prod_{\beta\in\Rgot'_+}(1-e^{-\beta})&=&
\sum_{\mu\in\Lambda'_+}m_\lambda(\mu)\,\chi_\mu^{G'}\prod_{\beta\in\Rgot'_+}(1-e^{-\beta})\\
&=&\sum_{\mu\in\Lambda'_+}\sum_{w'\in W'}(-1)^{w'} m_\lambda(\mu)\, e^{w'\bullet \mu}.
\end{eqnarray*}
Finally, we obtain the identity
$$
\sum_{\xi'\in\Lambda'} N_\lambda(\xi') e^{\xi'}=\sum_{\mu\in\Lambda'_+}
\sum_{w\in W'}(-1)^{w'} m_\lambda(\mu)\, e^{w'\bullet\mu},
$$
that shows two things:
\begin{itemize}
    \item $N_\lambda(\mu)=m_\lambda(\mu)$ if $\mu$ is dominant,
    \item $N_\lambda(w'\bullet\xi')=(-1)^{w'} N_\lambda(\xi')$, for every 
    $(w',\xi')\in W'\times \Lambda'$.
\end{itemize}

At this stage, we have proved Kostant's multiplicity formula. In the following we rewrite this formula in another form. Let's consider the following $T'$-module  
$$
\ngot_\Sigma:=\bigoplus_{\beta\in\Sigma} \Cbb_{-\beta}.
$$

\begin{prop}\label{prop:kostant-formula}
    For any $\lambda\in\Lambda_+$, we have the following 
    restriction formula
    \begin{equation}\label{eq:kostant-2}
     V_\lambda^G\vert_{G'}=\frac{1}{\# W'}\sum_{w\in W}(-1)^w \, 
   {\rm Hol}_{T'}^{G'}
    \left(\Cbb_{\overline{w\bullet \lambda}}\otimes{\rm Sym}(\ngot_\Sigma)\right).   
    \end{equation}
\end{prop}

{\em Proof:} Since ${\rm Sym}(\ngot_\Sigma)=\sum_{\xi'\in\Lambda'}\Pcal_\Sigma(\xi')\Cbb_{-\xi'}$ we have
\begin{eqnarray*}
\sum_{w\in W}(-1)^w \,\Cbb_{\overline{w\bullet\lambda}}\otimes{\rm Sym}(\ngot_\Sigma)
&=&
\sum_{w\in W}\sum_{\xi'\in\Lambda'}(-1)^w \,\Pcal_\Sigma(\xi')\,
\Cbb_{\overline{w\bullet \lambda}-\xi'}\\
&=& \sum_{\xi'\in\Lambda'}N_\lambda(\xi') \,\Cbb_{\xi'}.
\end{eqnarray*}
Hence the right hand side of (\ref{eq:kostant-2}) is equal to 
$\frac{1}{\# W'}\sum_{\xi'\in\Lambda'}
N_\lambda(\xi') {\rm Hol}_{T'}^{G'}\left(\Cbb_{\xi'}\right)$. We use now the following facts:
\begin{itemize}
    \item[--] $N_\lambda(w'\bullet\xi') {\rm Hol}_{T'}^{G'}
    \left(\Cbb_{w'\bullet\,\xi'}\right)=
    N_\Sigma(\xi') {\rm Hol}_{T'}^{G'}\left(\Cbb_{\xi'}\right)$ for every $(w',\xi')\in W'\times \Lambda'$.
    \item[--] $N_\lambda(\xi') {\rm Hol}_{T'}^{G'}\left(\Cbb_{\xi'}\right)=0$ if $\xi'\notin W'\bullet \Lambda'_+$.
    \item[--] $N_\lambda(\mu) {\rm Hol}_{T'}^{G'}\left(\Cbb_{\mu}\right)=m_\lambda(\mu) V^{G'}_\mu$ if $\mu\in \Lambda'_+$.
\end{itemize}
We have completed the proof of (\ref{eq:kostant-2}). $\Box$

\medskip

We conclude this section with a few examples.

\subsection{$K\subset K\times K$}
Let $K$ be a connected compact Lie group. Here we work with the Lie group $G = K\times K$ containg $K$ diagonally. Here $\Sigma\subset\tgot^*$ is equal to the set $\Rgot_+$ of positive roots for $K$. 
We denote by $\Pcal: \Lambda\to \Nbb$ the partition function associated to the set $\Rgot_+$. 

If $\lambda,\mu,\nu$ are three dominant weights, we denote by $c_{\lambda,\mu}^\nu$ the multiplicity of $V_\nu^K$ in $V_\lambda^K\otimes V_\mu^K$. The branching formula (\ref{eq:kostant-1}) becomes
$$
c_{\lambda,\mu}^\nu=\sum_{w_1,w_2\in W}(-1)^{w_1w_2}\, 
\Pcal(w_1\bullet\lambda+w_2\bullet\mu-\nu).
$$
This formula was first observed by Steinberg \cite{Steinberg-61}.

Let's take a closer look at the branching formula (\ref{eq:kostant-2}). The $T$-module $\ngot_\Sigma$ is equal to $\ngot:=\sum_{\alpha>0}\Cbb_{-\alpha}$. By definition of the holomorphic induction map ${\rm Hol}_{T}^{K}$, we have 
$$
    {\rm Hol}_{T}^{K}
    \left(\Theta\otimes{\rm Sym}(\ngot)\right)=(-1)^d \, {\rm Ind}_{T}^{K}
    \left(\Theta\otimes \Cbb_{2\rho}\right)
$$

for any $\Theta\in R(T)$, with $d=\frac{1}{2}\dim K/T$. Finally, (\ref{eq:kostant-2}) becomes
\begin{eqnarray*}
   V_\lambda^K\otimes V_\mu^K\vert_K &=& 
   \frac{(-1)^d}{\# W} \sum_{w_1,w_2\in W}(-1)^{w_1w_2}
{\rm Ind}_{T}^{K}\left(\Cbb_{w_1(\lambda+\rho)+ w_2(\mu+\rho)}\right) \\
&=& (-1)^d \sum_{w\in W}(-1)^{w}
{\rm Ind}_{T}^{K}\left(\Cbb_{w(\lambda+\rho)+ \mu+\rho}\right).
\end{eqnarray*}
 The latter formula is also obtained in (\ref{eq:general-Clebsch-Gordan}).

\subsection{$U(p)\times U(q)\subset U(n)$}

In this example the torus $T$ of diagonal matrices is the maximal torus for both $U(n)$ and the subgroup $U(p)\times U(q)$. Here the $T$-module $\ngot_\Sigma$ is the $T$-restriction of the $U(p)\times U(q)$-module $(\Cbb^p)^*\otimes\Cbb^q$, and the quotient $W'\backslash W$ is isomorphic to the subset ${\rm Shuffle}(p,q)$ formed by the elements $w\in \Sgot_{n}$ satisfying $w(1)<\cdots<w(p)$ and $w(p+1)<\cdots<w(p+q)$. Here, the branching formula (\ref{eq:kostant-2}) gives
\begin{equation}\label{eq:kostant-U-p-q}
 V_\lambda^{U(n)}\vert_{U(p)\times U(q)}=\left(\sum_{w\in {\rm Shuffle}(p,q)} (-1)^w \,
{\rm Hol}_{T}^{U(p)\times U(q)}(\Cbb_{w\bullet\lambda})\right) \otimes 
{\rm Sym}((\Cbb^p)^*\otimes\Cbb^q).   
\end{equation}

Let's consider the case $q=1$. From (\ref{eq:kostant-U-p-q}), we derive the following branching formula for the restriction to the subgroup $U(n-1)$:
$$
V_\lambda^{U(n)}\vert_{U(n-1)}=\left(\sum_{k=1}^n (-1)^{n-k}\, V^{U(n-1)}_{\lambda[k]}\right) \otimes {\rm Sym}((\Cbb^{n-1})^*),
$$
with $\lambda[n]=(\lambda_1,\ldots,\lambda_{n-1})$ and $\lambda[k]=(\lambda_1,\ldots,\lambda_{k-1},\lambda_{k+1}-1,\ldots,\lambda_{n}-1)$ for $1\leq k\leq n-1$.

\begin{figure}[h]
    \centering
   \includegraphics[scale=2]{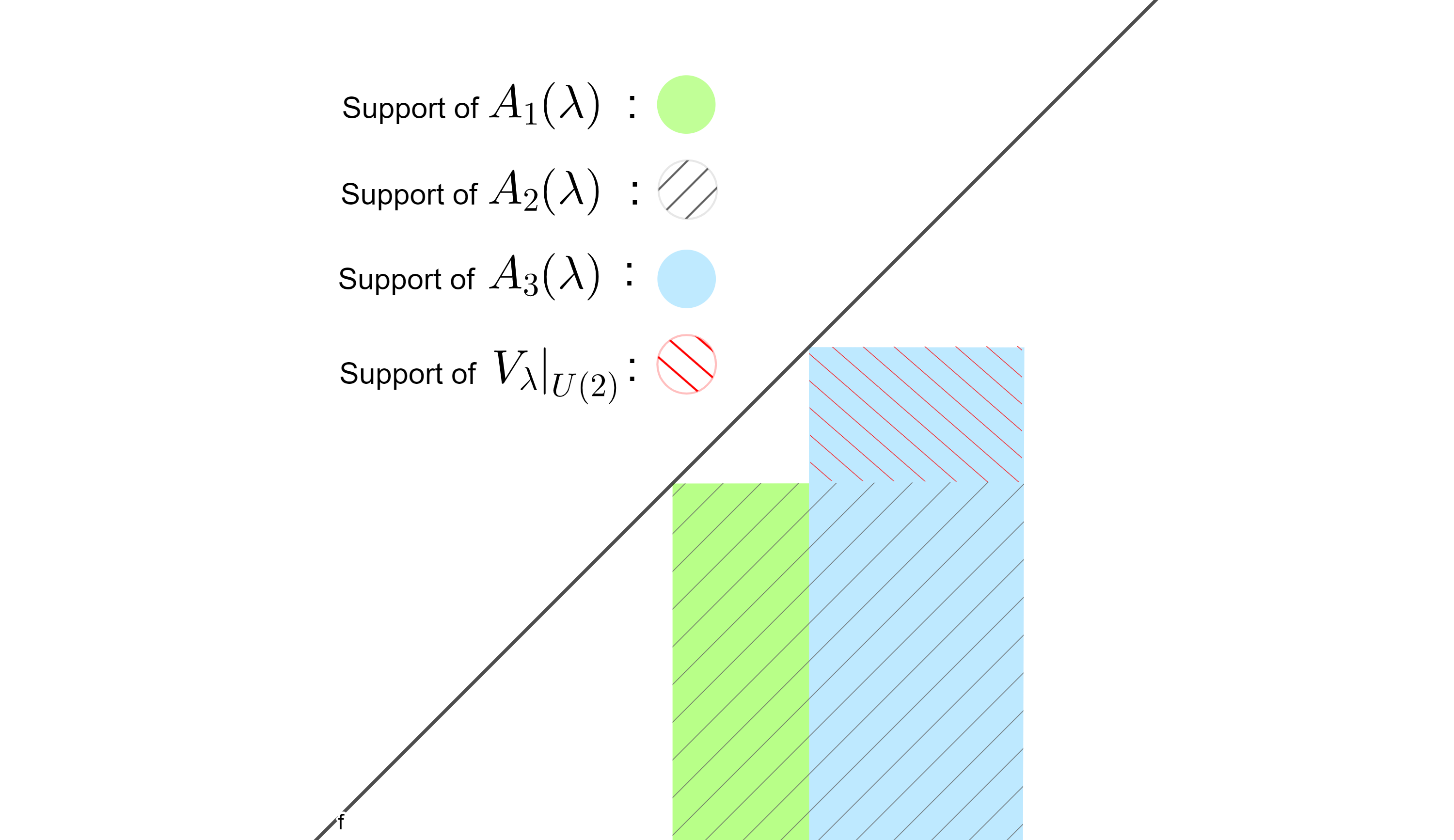} 
    \caption{Kostant decomposition}
    \label{fig:U(3)-Kostant}
\end{figure}

Let's consider the case $n=3$. For any $\lambda=(\lambda_1\geq\lambda_2\geq\lambda_3)$ we obtain the following formula
$$
V_\lambda^{U(3)}\vert_{U(2)}=\left(V^{U(2)}_{\lambda[3]}-V^{U(2)}_{\lambda[2]}+V^{U(2)}_{\lambda[1]}\right) \otimes {\rm Sym}((\Cbb^{2})^*).
$$
Hence $V_\lambda^{U(3)}\vert_{U(2)}=A_1(\lambda)-A_2(\lambda)+A_3(\lambda)$ with 
\begin{eqnarray*}
   A_3(\lambda) &=& V^{U(2)}_{\lambda[3]} \otimes {\rm Sym}((\Cbb^{2})^*)=
\sum_{\lambda_1\geq a\geq\lambda_2\geq b} V^{U(2)}_{(a,b)} ,\\
   A_2(\lambda) &=& V^{U(2)}_{\lambda[2]} \otimes {\rm Sym}((\Cbb^{2})^*)=
\sum_{\lambda_1\geq a\geq \lambda_3-1\geq b} V^{U(2)}_{(a,b)},\\
   A_1(\lambda) &=& V^{U(2)}_{\lambda[1]} \otimes {\rm Sym}((\Cbb^{2})^*)=
\sum_{\lambda_2-1\geq a\geq \lambda_3-1\geq b} V^{U(2)}_{(a,b)}.
\end{eqnarray*}
We recover the classical branching formula $V_\lambda^{U(3)}\vert_{U(2)}:=\sum_{\lambda_1\geq a\geq\lambda_2\geq b\geq \lambda_3} V^{U(2)}_{(a,b)}$ (see \cite{Goodman-Wallach}, section 8.1). In Figure \ref{fig:U(3)-Kostant}, one can visualize the support of each characters $A_k(\lambda)$.


\end{document}